\documentclass{amsart}
\usepackage{amsmath}
\usepackage{amsfonts}
\usepackage{amssymb}

\usepackage{hyperref}
\usepackage{tikz}

\definecolor{shaded}{rgb}{0.99, 0.76, 0.8} 

\usepackage{amsthm}
\theoremstyle{definition}
\newtheorem{thm}{Theorem}[section]
\newtheorem{defn}[thm]{Definition}
\newtheorem{ex}[thm]{Example}
\newtheorem{rem}[thm]{Remark}
\newtheorem{prop}[thm]{Proposition}
\newtheorem{cor}[thm]{Corollary}
\newtheorem{lem}[thm]{Lemma}

\DeclareMathOperator{\id}{id}

\newcommand{\Z}{\mathbb{Z}}

\newcommand{\be}{\mathbf{e}}

\renewcommand{\bar}{\overline}

\newcommand{\drawislandnine}[9]{
\vcenter{\hbox{
\begin{tikzpicture}[scale=.5,every node/.style={scale=.7}]
\draw[thick, shift={(-.5,-.5)}] (0,0) rectangle (5,5);
\draw[densely dotted, shift={(-.5,-.5)}] (0,0) grid (5,5);
\fill[shaded, shift={(-.5,-.5)}] (2,2) rectangle (3,3);
\node at (1,3) {$#1$};
\node at (1,2) {$#2$};
\node at (1,1) {$#3$};
\node at (2,1) {$#4$};
\node at (3,1) {$#5$};
\node at (3,2) {$#6$};
\node at (3,3) {$#7$};
\node at (2,3) {$#8$};
\node at (2,2) {$#9$};
\foreach \x in {0,...,4} {
 \node at (\x,0) {.};
 \node at (\x,4) {.};
 \node at (0,\x) {.};
 \node at (4,\x) {.};
}
\end{tikzpicture}
}}
}
\newcommand{\drawislandeight}[8]{\drawislandnine{#1}{#2}{#3}{#4}{#5}{#6}{#7}{#8}{1}}

\newcommand{\drawisland}[4]{\drawislandeight{#1}{#1}{#2}{#2}{#3}{#3}{#4}{#4}}

\newcommand{\drawtallisland}[1]{
\def\numberone{#1}
\dtictd
}
\newcommand{\dtictd}[9]{
\vcenter{\hbox{
\begin{tikzpicture}[scale=.5,every node/.style={scale=.7}]
\draw[thick, shift={(-.5,-.5)}] (0,0) rectangle (5,6);
\draw[densely dotted, shift={(-.5,-.5)}] (0,0) grid (5,6);
\fill[shaded, shift={(-.5,-.5)}] (2,2) rectangle (3,4);
\node at (1,4) {$\numberone$};
\node at (1,3) {$#1$};
\node at (1,2) {$#2$};
\node at (1,1) {$#3$};
\node at (2,1) {$#4$};
\node at (3,1) {$#5$};
\node at (3,2) {$#6$};
\node at (3,3) {$#7$};
\node at (3,4) {$#8$};
\node at (2,4) {$#9$};
\node at (2,2) {$1$};
\node at (2,3) {$1$};
\foreach \x in {0,...,4} {
 \node at (\x,0) {.};
 \node at (\x,5) {.};
 \node at (0,\x) {.};
 \node at (4,\x) {.};
}
\end{tikzpicture}
}}}

\begin{document}

\title[Digital Second Homotopy Group]{A Second Homotopy Group for Digital Images}

\author[G.\ Lupton]{Gregory Lupton}
\author[O.\ Musin]{Oleg Musin}
\author[N.\ Scoville]{Nicholas A. Scoville}
\author[P.C.\ Staecker]{P.~ Christopher Staecker}
\author[J.\ Trevino]{Jonathan Trevi{\~n}o-Marroqu{\'i}n}

\address{Department of Mathematics and Statistics, Cleveland State University, Cleveland OH 44115 U.S.A.}

\email{g.lupton@csuohio.edu}

\address{
School of Mathematical and Statistical Sciences , The University of Texas Rio Grande Valley}

\email{oleg.musin@utrgv.edu}

\address{Department of Mathematics, Computer Science, and Statistics, Ursinus College, Collegeville PA 19426 U.S.A.}

\email{nscoville@ursinus.edu}

\address{Department of Mathematics, Fairfield University, Fairfield CT 06824 U.S.A.}

\email{cstaecker@fairfield.edu}

\address{Centro de Investigaci{\'o}n en Matem{\'a}ticas, A.C., Guanajuato GTO 36023 M{\'e}xico}

\email{jonathan.trevino@cimat.mx}

\date{\today}

\keywords{Digital Image, Digital Topology,  Trivial Extension, Subdivision,  Digital Second Homotopy Group, Digital Sphere}
\subjclass[2010]{ (Primary) 55Q99;  (Secondary) 68U10 68R99}

\thanks{The work presented here is a direct result of working sessions held by the authors during the AIM workshop \emph{Discrete and Combinatorial Homotopy Theory} held in San Jose, CA the week of March 13--17, 2023.  We thank the AIM and the organizers for providing a stimulating environment and enabling our collaboration. It was at the workshop that we became aware of the graph-theoretic work mentioned at several points in the paper.}

\begin{abstract}   We define a second (higher) homotopy group for digital images. Namely, we construct a functor from digital images to abelian groups, which
closely resembles the ordinary second homotopy group from algebraic topology.    We illustrate that our approach can be effective by computing this (digital) second homotopy group for a digital 2-sphere.
\end{abstract}

\maketitle

\section{Introduction}

\emph{Digital topology} refers to the use of notions and methods from (algebraic) topology to study digital images. A digital image in our sense is an idealization of an actual digital image which consists of  pixels in the plane, or higher dimensional analogues of such. Specifically, a \emph{digital image} is a subset of $\Z^n$ together with  some chosen ``adjacency relation'' induced by the integer lattice. 
 The aim of digital topology is to provide useful theoretical background for certain steps of image processing, such as  contour filling, border and boundary following, thinning, and feature extraction or recognition (e.g. see p.273 of \cite{Ko-Ro91}).  There is an extensive literature on digital topology (see e.g. \cite{Ro86, Ko-Ro91, Bo99, Evako2006} and the references therein).

A number of authors have studied the fundamental group in the setting of digital topology (see \cite{Kong89, Bo99, LOS19c, LuSc22} for a sample).
In this paper, we begin a development of higher homotopy groups in the digital topology setting, essentially extending the approach of \cite{LOS19c, LuSc22} to an initial study of  the second homotopy group of a digital image.  In subsequent work, we hope to continue the development begun here into a fuller treatment of higher homotopy groups of digital images. These notions do appear in the digital topology literature \cite{MVK14, VeKa15}, but these earlier treatments differ from ours, as we describe in Remark \ref{turkishremark}.   Higher homotopy groups have also been developed in a graph-theoretic setting that is very closely related to the setting in which we work (see \cite{BKL01, Do08, C-S21} for a sample of work in this area).   Our emphasis on the operations of \emph{trivial extension} and \emph{subdivision} of maps between digital images separates our work from this graph-theoretic work because these operations rely on specifics of the orthogonal integer lattice and do not naturally generalize to arbitrary graphs. Also, our present treatment is more computational: the second half of the paper focuses on a very hands-on calculation of a nontrivial second homotopy group.



The paper is organized as follows.  In Section \ref{sec: extension} we give basic definitions and define the fundamental notion of \emph{extension homotopy}. In Section~\ref{sec: subd and doubling}, we discuss a more general notion of column- or row-doubling, and subdivision of a map on a digital image. 
In Section~\ref{sec: defn of pi_2} we define the second homotopy group $\pi_2(X,x_0)$, show that it is an abelian group (Thm.~\ref{thm: group} and Thm.~\ref{thm: abelian}), and establish some basic properties such as independence of basepoint (Prop.~\ref{prop: basepoint}) and functoriality (Prop.~\ref{prop: functor}).  We take the general development far enough to establish behaviour with respect to products (Thm.~\ref{thm: products}).

The remainder of the paper is devoted to a calculation of the second homotopy group of a digital $2$-sphere.  Perhaps unsurprisingly, we find that this group is isomorphic to $\Z$.  But the way in which we calculate this to be so involves some interesting combinatorial ingredients.  Notably, in Section~\ref{sec: degree} we develop a \emph{triangle-counting function} for a map of the kind that represents an element of the homotopy group.  This may be conceived of as the degree of such a map, and it may be determined directly from the map in a very transparent way.  We complete the calculation of the second homotopy group of a digital $2$-sphere in Section~\ref{sec: flooding}.  A short Section~\ref{sec: future} ends the paper with some suggestions for future work.

\section{Homotopy, trivial extension, and extension-homotopy}\label{sec: extension}
\begin{defn}
A \emph{digital image} $X\subset \Z^n$ is a finite subset of the integer lattice, together with a chosen reflexive \emph{adjacency relation} denoted $\sim$. 

Typical choices of the adjacency relation are the various adjacencies denoted $c_i$ for $i \in \{1,\dots,n\}$, in which $x\sim y$ when the coordinates of $x$ and $y$ differ by at most 1 in at most $i$ positions, and are equal in all other positions. These $c_i$ adjacencies follow the lattice structure of $\Z^n$, with different interpretations of diagonal adjacencies. The $c_1$ relation includes no diagonally adjacent points, while the $c_n$ adjacency counts any diagonal points as adjacent.
In $\Z^2$, the $c_1$ adjacency is referred to as ``4-adjacency'', because each point is adjacent to 4 points other than itself, while the $c_2$ adjacency is referred to as ``8-adjacency''.

In the digital topology literature, the adjacency relation is often taken to be antireflexive, so that a point is not adjacent to itself. In our case, though, we follow \cite{LOS19c, LuSc22} and require our relation to be reflexive, which simplifies some definitions and clarifies the connections to related work in graph theory.
\end{defn}

Let $I_n = \{0,\dots, n\}$ be the \emph{digital interval}, considered with the usual adjacency in which $a \sim b$ if and only if $|a-b|\le 1$. 

For a real interval $[n,m]$, define $[n,m]_{\Z}=[n,m]\cap \Z$, so that $I_n = [0,n]_\Z$. We will also consider products of the form $I_{m,n} = I_{m} \times I_{n}$, which we refer to as \emph{rectangles}. For a rectangle $I_{m,n}$, let $\partial I_{m,n}$ be the \emph{boundary}, defined by:
\[ \partial I_{m,n} = (\{0,m\} \times I_{n}) \cup (I_{m} \times \{0,n\}). \]

For products of digital images, we always use the categorical product adjacency.  That is, if $x, x' \in X$ and $y, y' \in Y$, then $(x, y) \sim (x', y') \in X \times Y$ if and only if $x \sim x'$ and $y \sim y'$. In the case of $I_{m,n} \subset \Z^2$, we have $(a,b) \sim (a',b')$ if and only if $|a-a'| \le 1$ and $|b-b'|\le 1$. This choice of product essentially dictates that we always use the $c_2$ adjacency (or $8$-adjacency) on the rectangle $I_{m,n}$. 

\begin{defn}\label{def: continuous}
For two digital images $X,Y$, a function $f:X\to Y$ is \emph{[digitally] continuous} when $x\sim x'$ implies $f(x)\sim f(x')$ for all $x,x' \in X$.

It is easy to see that the composition of two continuous functions is continuous.
\end{defn}

There is a natural interpretation of this setup in the context of graph theory, and in fact a body of very similar work has developed in graph theory, independent from and until recently unknown to the digital topology community. Notably the areas of A-theory (see \cite{BKL01, Bar-La05, B-B-L-L06}) and $\times$-homotopy theory (see \cite{Do08, Do09, C-S21}). From the graph theoretic point of view, any digital image $X$ may be regarded as an induced (reflexive) subgraph of the integer lattice $\Z^n$. Then a digitally continuous function $f:X\to Y$ is simply a graph homomorphism, provided that we represent $X$ and $Y$ as reflexive graphs (that is, we must have a looped edge at every vertex in the graph). These looped edges must be present in the codomain to allow the map $f$ to collapse an edge to a vertex and yet to map edges to edges. For example if $a\sim b\in X$ with $a\neq b \in X$ and $f(a)=f(b)=c$, then there must be a looped edge at $c$ in order for $f$ to carry the edge $(a,b)$ to an edge in $Y$.

In abstract terms, the category of digital images and digitally continuous functions is the same as the category of reflexive graphs and graph homomorphisms. Thus in many cases, the constructions used in A-theory and $\times$-homotopy theory of e.g.\ \cite{BKL01, Do08, C-S21} are the same as the constructions in the digital topology literature, which developed independently.

The differences between A-theory and $\times$-homotopy theory and the digital theory arise in different choices made in the definition of homotopy:

\begin{defn}\label{homotopydef}
Two continuous maps $f,g:X \to Y$ are \emph{[digitally] homotopic} if there is some $k$ with a continuous map $H:X\times I_k \to Y$ with $H(x,0) = f(x)$ and $H(x,k)=g(x)$ for all $x$. In this case we write $f \simeq g$. \end{defn}

This notion of homotopy gives an equivalence relation on the set of all maps $X \to Y$ (see Lemma 3.16 of \cite{LOS19a}, for example).
The choice of product in the definition of homotopy has an important effect on developments. As noted above, we use the categorical product, which leads to the homotopy notion typically featured in the $\times$-homotopy theory of \cite{Do08}. The development in A-theory \cite{BKL01} has traditionally used the ``box product'', in which $(x,t)\sim (y,s)$ if and only if either $x\sim y$ and $t=s$, or $x=y$ and $t\sim s$, which leads to a weaker notion of homotopy often referred to as the \emph{box homotopy}. 

The digital topology literature following Boxer \cite{Bo99} has typically used the word ``homotopy'' to indicate the box homotopy, though some papers have explored the categorical product homotopy: it is called ``strong homotopy'' in \cite{St21}. The present paper follows the terminology used in \cite{LOS19c, LuSc22}, with the relation of Definition \ref{homotopydef} simply called ``homotopy.''

If $k=1$ in Definition~\ref{homotopydef}, then we refer to the homotopy as a \emph{one-step homotopy}.  That is, a one-step homotopy between maps $f, g\colon X \to Y$ is a continuous map $H\colon X \times I_1 \to Y$ with $H(x, 0) = f(x)$ and $H(x, 1) = g(x)$. In this case, we say that $f$ and $g$ are \emph{one-step homotopic}. 
There is a simple criterion for maps to be one-step homotopic:

\begin{lem}\label{onesteplem}
\cite[Theorem 2.4]{St21} Suppose continuous maps $f,g\colon X\to Y$ satisfy $f(x)\sim g(x')$ in $Y$  whenever $x\sim x'$ in $X$.  Then $f$ and $g$ are one-step homotopic.  Indeed, the homotopy $H\colon X \times I_1 \to Y$ defined by $H(x, 0) = f(x)$ and $H(x, 1) = g(x)$ is a one-step homotopy from $f$ to $g$.\end{lem}

\begin{proof}
We need to confirm continuity of $H$.  That is, we require $H(x, t) \sim H(x', t')$ in $Y$ whenever we have $(x, t) \sim (x', t')$ in $X \times I_1$ (recall that we are using the categorical product for adjacencies in $X \times I_1$).  If $t'= t \in I_1$, then the adjacencies in $Y$ follow from continuity of $f$ or $g$.  If $t' \not= t \in I_1$, then the hypothesis provides the required adjacencies.   
\end{proof}

In fact, the hypothesis on $f$ and $g$ in Lemma~\ref{onesteplem} gives a characterization of when maps $f$ and $g$ are one-step homotopic, but we will not need the converse here. 
We will make repeated use of the following simple kind of one-step homotopy.

\begin{lem}\label{lem: spider move}
Let $f:X\to Y$ be a continuous map of digital images, with  $f(a) = b$ at some particular $a \in X$. Suppose $b' \in Y$ is adjacent to $b$ and also adjacent to $f(a')$ for every $a' \in X$ adjacent to $a$.
Define a map 
\[ g(x) = \begin{cases} f(x) & x \not= a \\ b' & x=a \end{cases} \]
Then $g$ is continuous and one-step homotopic to $f$ via the one-step homotopy $H \colon X \times I_1 \to Y$ with $H(x, 0) = f(x)$ and $H(x, 1) = g(x)$.
\end{lem}

\begin{proof}
Continuity of $g$ follows because (continuous) $f$ and $g$ agree apart from at $a \in X$, where we have $g(a) = b' \sim f(x) = g(x)$ for all $x \sim a$ (but not equal to $a$)  in $X$.  Then the maps  $f(x)$ and $g(x)$ satisfy the condition of Lemma~\ref{onesteplem}.  
\end{proof}

\begin{defn}\label{spidermove}
The type of one-step homotopy in Lemma~\ref{lem: spider move}, whereby a map is changed in value at a single point, is called  a \emph{spider move}. 
\end{defn}

The following theorem, which has appeared in both the digital topology literature and the $\times$-homotopy literature, shows that an arbitrary homotopy can be realized by a finite sequence of spider moves.  
\begin{thm}\label{spidermovethm}\cite[Proposition 4.4]{C-S21} \cite[Theorem 3.2]{St21} 
For continuous $f,g\colon X\to Y$, the maps $f$ and $g$ are homotopic if and only if they are homotopic by a finite sequence of spider moves. \qed
\end{thm}

For positive integers $m,n$ and a pointed digital image $(X,x_0)$, we will consider continuous maps of pairs $f\colon (I_{m,n},\partial I_{m,n}) \to (X,x_0)$. That is, continuous maps $f$ with $f (\partial I_{m,n}) = \{x_0\}$. 

It is often convenient to visualize a function $f:(I_{m,n},\partial I_{m,n}) \to (X,x_0)$ as a labeling of the points of the rectangle $I_{m,n}$ with labels taken from the set $X$. For example a function $f:(I_{4,4},\partial I_{4,4}) \to (X,x_0)$ would be represented by the labeled rectangle in Figure \ref{labelingfig}. For simplicity in our pictures, we will indicate the label of the basepoint with a dot.

\begin{figure}
\[
\begin{tikzpicture}[scale=.5,every node/.style={scale=.7}]
\draw[thick, shift={(-.5,-.5)}] (0,0) rectangle (5,5);
\draw[densely dotted, shift={(-.5,-.5)}] (0,0) grid (5,5);
\node at (1,3) {$x_1$};
\node at (1,2) {$x_4$};
\node at (1,1) {$x_7$};
\node at (2,1) {$x_8$};
\node at (3,1) {$x_9$};
\node at (3,2) {$x_6$};
\node at (3,3) {$x_3$};
\node at (2,3) {$x_2$};
\node at (2,2) {$x_5$};
\foreach \x in {0,...,4} {
 \node at (\x,0) {.};
 \node at (\x,4) {.};
 \node at (0,\x) {.};
 \node at (4,\x) {.};
}
\end{tikzpicture}
\]
\caption{Representation of a typical map $f:(I_{4,4}, \partial I_{4,4}) \to (X,x_0)$ for $x_i\in X$. Each pixel is labeled with its function value, so that e.g.\ $f(2,1) = x_8\in X$. Values representing the base point $x_0$ are labeled with a dot.\label{labelingfig}}
\end{figure}

Our second homotopy group is modeled on homotopy classes of maps of pairs $(I_{m,n},\partial I_{m,n}) \to (X,x_0)$, where the homotopies preserve values at the boundary in the following sense:

Given $f,g\colon (I_{m,n},\partial I_{m,n}) \to (X,x_0)$ a homotopy $H\colon I_{m,n}\times I_k \to X$, we say $H$ is a \emph{homotopy relative to the boundary} when $H(\partial I_{m,n} \times I_k) = \{x_0\}$.

In our development, we often encounter a situation in which we have a ``local" homotopy that only involves values of a map in some part of the rectangle $I_{m,n}$.  If such a homotopy is stationary on the boundary of a subrectangle, then it may be extended to a homotopy of the whole rectangle in an obvious way.  

\begin{lem}\label{localized homotopy}
Let  $R$ be a subrectangle of $I_{m,n}$. Let $f:(I_{m,n}, \partial I_{m,n}) \to (X, x_0)$ be a map for which $f_R$, the restriction of $f$ to the subrectangle $R$, is  a map of pairs $f_R: (R, \partial R) \to (X, x_0)$. Let $g:(R, \partial R) \to (X, x_0)$ be a map on the subrectangle such that $f_R\simeq g$ by a homotopy that is stationary on the boundary $\partial R$.
Then the map $A: (I_{m,n}, \partial I_{m,n}) \to (X, x_0)$ defined by:
\[ A(a, b) = \begin{cases} g(a, b) &\text{ if } (a, b) \in R \\
f(a, b) &\text{ if } (a, b)\not \in R \\
\end{cases} \]
is continuous and we have $f \simeq A$ by a homotopy relative to the boundary.
\end{lem}

\begin{proof}
Continuity of the map $A$ is assured because the maps $f$ and $g$ agree on the boundary of the rectangle $\partial R$, which separates $I_{m,n}$ into non-adjacent interior and exterior:  A point interior to the rectangle $R$ cannot be adjacent to a point exterior to the rectangle $R$.  Thus, $A(a, b) \sim A(a', b')$ follows from continuity of $g$ for points $(a, b) \sim (a', b')$ in $R$ (including its boundary), and from continuity of $f$ for points $(a, b) \sim (a', b')$ in $R^C \cup \partial R$. (Here $R^C$ denotes the complement of $R$ in $I_{m,n}$.)

Suppose $G\colon R \times I_k \to X$ is the homotopy---stationary on $\partial R$---from $f_R$ to $g$.  Then the homotopy $H: I_{m,n}\times I_k \to X$ defined by
\[ H\left( (a, b), t \right) = \begin{cases} G((a, b),t) &\text{ if } (a, b) \in R \\
f(a, b) &\text{ if } (a, b) \not \in R \\
\end{cases} \]
starts at $f$ and ends at $A$.  This homotopy is continuous by reasoning similar to that of the first part.  Namely, $\partial R \times I_k$ separates $I_{m,n} \times I_k$ into non-adjacent interior and exterior.  Then
$H\left( (a, b), t \right) \sim H\left( (a', b'), t' \right)$ follows from continuity of $G$ for points $\left( (a, b), t \right) \sim \left( (a', b'), t' \right)$ in $R\times I_k$ and from continuity of $f$ for points $\left( (a, b), t \right) \sim \left( (a', b'), t' \right)$ in $(R^C \cup \partial R)\times I_k$.
\end{proof}

Boxer's definition of the fundamental group in \cite{Bo99} uses a construction which he calls \emph{trivial extension} of a loop. We adapt this concept for our higher dimensional setting by simply repeating values of the base point.

We define \emph{trivial extensions} of maps $(I_{m,n},\partial I_{m,n}) \to (X,x_0)$ as follows: if $m' \ge m$ and $n'\ge n$, then there is a natural inclusion $I_{m,n} \subseteq I_{m',n'}$. We say 
$\bar f\colon (I_{m',n'},\partial I_{m',n'}) \to (X,x_0)$ is a trivial extension of $f \colon (I_{m,n},\partial I_{m,n}) \to (X,x_0)$ when
\[ \bar f(x) = \begin{cases}
f(x) & \text{ if $x\in I_{m,n}$,} \\
x_0 & \text{ otherwise.}
\end{cases} \]
See Figure~\ref{trivialfig} for a pictorial representation of a trivial extension.

\begin{figure}
\[
f:
\vcenter{\hbox{
\begin{tikzpicture}[scale=.4]
\foreach \x in {0,...,4} {
 \foreach \y in {0,...,4} {
   \node at (\x+.5,\y+.5) {.};
  }
 }
\draw[thick] (0,0) rectangle (5,5);
\draw[densely dotted] (0,0) grid (5,5);
\draw[fill=white] (1,1) rectangle (4,4);
\node at (2.5,2.5) {$*$};
\end{tikzpicture}
}}
\qquad
\bar f:
\vcenter{\hbox{
\begin{tikzpicture}[scale=.4]
\foreach \x in {0,...,8} {
 \foreach \y in {0,...,5} {
   \node at (\x+.5,\y+.5) {.};
  }
 }
\draw[thick] (0,0) rectangle (9,6);
\draw[densely dotted] (0,0) grid (9,6);
\draw[fill=white] (1,1) rectangle (4,4);
\node at (2.5,2.5) {$*$};
\end{tikzpicture}
}}
\]
\caption{Schematic of a map $f:I_{4,4} \to X$ with its trivial extension $\bar f \colon I_{8,5} \to X$.\label{trivialfig}}
\end{figure}

Whereas homotopy is a relation between maps with the same domain, we will need to compare maps whose domains are differently sized rectangles.  We do this through the following device. 

\begin{defn}\label{def: extn homotopy}
Given two maps $f:(I_{m,n},\partial I_{m,n}) \to (X,x_0)$ and $g: (I_{m',n'},\partial I_{m',n'}) \to (X,x_0)$, we write $f \approx g$, and say that $f$ and $g$ are \emph{extension-homotopic}, when there exist $\bar m \ge \max(m,m')$ and $\bar n \ge \max(n,n')$ and $\bar f, \bar g: I_{\bar m, \bar n} \to X$ with $\bar f$ a trivial extension of $f$ and $\bar g$ a trivial extension of $g$ and $\bar f$ homotopic to $\bar g$ by a  homotopy relative to the boundary.
\end{defn}

\begin{thm}\label{thm: extn-htpy equiv}
Extension homotopy of maps is an equivalence relation on the set of maps $(I_{m,n},\partial I_{m,n}) \to (X,x_0)$ for all sizes of rectangles.
\end{thm}

Reflexivity and symmetry follow immediately because homotopy (relative the boundary) of maps is an equivalence relation. Transitivity is a consequence of the following simple lemma: 

\begin{lem}\label{lem: t.e. htpy}
Suppose we have maps $f \simeq g \colon (I_{m,n},\partial I_{m,n}) \to (X,x_0)$ homotopic relative to the boundary.  Let $f',  g' \colon (I_{m',n'},\partial I_{m',n'}) \to (X,x_0)$ be trivial extensions of $f$ and $g$ to the same-sized rectangle, for $m' \geq m$ and $n' \geq n$.  Then we have $f' \simeq g'$.  
\end{lem}

\begin{proof}
The proof is fairly obvious.  A homotopy $H \colon I_{m,n}\times I_{k} \to X$ from $f$ to $g$ relative the boundary extends to a homotopy $\overline{H} \colon I_{m',n'}\times I_{k} \to X$ from $f'$ to $g'$ relative the boundary, by setting $\overline{H}$ to be stationary at $x_0$ on all points of $I_{m',n'}$ not in $I_{m,n}$. This extension $\overline{H}$ is easily seen to be a continuous map on $I_{m',n'}\times I_{k}$ since a point of $I_{m',n'}$ not in $I_{m,n}$ cannot be adjacent to a point in the interior of $I_{m,n}$, and the original $H$ is already stationary at $x_0$ on all points of $\partial I_{m, n}$.
\end{proof}

\begin{proof}[Proof of Theorem~\ref{thm: extn-htpy equiv}]
As observed above, we only need show transitivity.  So, suppose we have maps $f_t\colon (I_{m_t,n_t}, \partial I_{m_t,n_t})\to (X,x_0)$ for $t = 1, 2, 3$ and that $f_1\approx f_2 \approx f_3$.  
Since $f_1\approx f_2$ there are $m'\geq \max\{m_1,m_2\}$ and $n'\geq \max\{n_1,n_2\}$ along with trivial extensions $f_1', f_2'\colon I_{m',n'}\to X$ of $f_1$ and $f_2$ and a homotopy from $f_1'$ to $f_2'$. Similarly, since $f_2\approx f_3$ there are $m''\geq \max\{m_2,m_3\}$ and $n''\geq \max\{n_2,n_3\}$ along with trivial extensions $f_2'', f_3''\colon I_{m'',n''}\to X$ of $f_2$ and $f_3$ and a homotopy from $f_2''$ to $f_3''$.

Let $\bar{m}=\max\{m',m''\}$ and $\bar{n}=\max\{n',n''\}$, and let $\bar{f_1}, \bar{f_2}, \bar{f_3}\colon I_{\bar{m},\bar{n}}\to X$ be trivial extensions of $f_1$, $f_2$ and $f_3$ respectively.  Then $\bar{f_1}$ and $\bar{f_3}$ are trivial extensions of $f_1'$ and $f_3''$, respectively, whilst $\bar{f_2}$ is a common trivial extension of both $f_2'$ and $f_2''$.  Hence we have $\bar{f_1} \simeq \bar{f_2}$ and $\bar{f_2} \simeq \bar{f_3}$ by Lemma~\ref{lem: t.e. htpy}.  Since homotopy is transitive, we have $\bar{f_1} \simeq \bar{f_3}$ and the result follows.
\end{proof}

\section{Column-Doubling and Row-Doubling}\label{sec: subd and doubling}
In this section we define column- and row-doubling operators and subdivisions of a digital map, and show their relations to trivial extensions.

\begin{defn}\label{def: alpha beta}
Given $i \in \{0,\dots, m\}$ and $j \in \{0,\dots, n\}$, let $\alpha_i: I_{m+1,n} \to I_{m,n}$ and $\beta_j: I_{m,n+1}\to I_{m,n}$ be the maps defined by
\[ \alpha_i(a,b) = \begin{cases} (a,b) &\text{ if $a\le i$}, \\
(a-1,b) &\text{ if $a> i$}.
\end{cases} \qquad
\beta_j(a,b) = \begin{cases} (a,b) &\text{ if $b\le j$}, \\
(a,b-1) &\text{ if $b> j$}.
\end{cases} 
\]
The map $\alpha_i$ simply omits one column of the domain, so the composition $f\circ \alpha_i: I_{m+1,n} \to X$ is a map which resembles $f$, but with column $i$ repeated once and all following columns shifted to the right by one position. Namely, viewing $f$ as a labeling of the points of $I_{m, n}$ with values from $X$, we have \emph{doubled the $i$th column} of labeled values to result in a similar labeling of $I_{m+1, n}$ for $f \circ \alpha_i$.  Likewise, $f\circ \beta_j$ is a map which resembles $f$, but with row $j$ doubled.
\end{defn}

\begin{thm}\label{columndoublehomotopy}
Let $f:(I_{m,n}, \partial I_{m,n}) \to (X,x_0)$ be continuous. We have homotopies relative the boundary $f \circ \alpha_m \simeq f \circ \alpha_{m-1} \simeq \cdots \simeq f \circ \alpha_0$ and $f \circ \beta_n \simeq f \circ \beta_{n-1} \simeq \cdots \simeq f \circ \beta_0$.  Consequently, we have
$f \approx f \circ \alpha_i$ for each $i \in \{0, \ldots, m\}$ and $f \approx f\circ \beta_j$ for each $j \in \{0, \ldots, n\}$.
\end{thm}

\begin{proof}
We will prove the statement for the $f\circ \alpha_i$. The proof for the $f\circ \beta_j$ is similar and we omit it.
Let $\bar f: I_{m+1,n} \to X$ be the trivial extension of $f$, which we may write as $\bar f = f\circ \alpha_m$.  We will show that  $f\circ \alpha_i \simeq f\circ \alpha_{i-1}$ for each $i \in \{1, \ldots, m\}$ and it follows that we have $f \approx f \circ \alpha_i$ for each $i$.

Indeed, we will use the criterion of Lemma \ref{onesteplem} to show $f\circ \alpha_i$ and $f\circ \alpha_{i-1}$ are one-step homotopic.  For $x \sim x' \in I_{m+1,n}$, we must show that 
$f\circ \alpha_i(x) \sim f\circ \alpha_{i-1}(x')$ in $X$.  Notice that $f\circ \alpha_i$ and $f\circ \alpha_{i-1}$ agree in value except at points $(i, b) \in I_{m+1,n}$: Unless one of $x$ or $x'$ has coordinates $(i, b)$ for some $b$, the desired conclusion follows from continuity of either map $f\circ \alpha_i$ or  $f\circ \alpha_{i-1}$. So, assume $x = (i, b)$ and $x' = (a', b')$ with $i \sim a'$ (and thus $a' \in \{ i-1, i, i+1\}$) and $b \sim b'$.  

If $a' = i-1$, we may use continuity of $f$ and the definitions of $f\circ \alpha_i$ and $f\circ \alpha_{i-1}$ to write
$$f\circ \alpha_i(x) = f\circ \alpha_i(i, b) = f(i, b) \sim f(i-1, b') = f\circ \alpha_{i-1}(i-1, b') = f\circ \alpha_{i-1}(x').$$
Similarly, if $a' = i$, we may write
$$f\circ \alpha_i(x) = f\circ \alpha_i(i, b) = f(i, j) \sim f(i-1, b') = f\circ \alpha_{i-1}(i, b') = f\circ \alpha_{i-1}(x').$$
Finally, if $a' = i+1$, we may write
$$f\circ \alpha_i(x) = f\circ \alpha_i(i, b) = f(i, b) \sim f(i, b') = f\circ \alpha_{i-1}(i+1, b') = f\circ \alpha_{i-1}(x').$$
It follows from Lemma \ref{onesteplem} that $f\circ \alpha_i$ and $f\circ \alpha_{i-1}$ are one-step homotopic, and the result follows. 
\end{proof}

Observe that, if $f\colon (I_{m,n}, \partial I_{m,n}) \to (X,x_0)$, then any trivial extension of $f$ may be obtained by repeatedly doubling the $m$th row and the $n$th column.
\begin{lem}\label{trivialdoubling}
Let $f\colon (I_{m,n}, \partial I_{m,n}) \to (X,x_0)$, and let $\bar f \colon (I_{r,s},\partial I_{r,s}) \to (X,x_0)$ be a trivial extension. Then we have
$$\bar f = f \circ \alpha^{r-m}_m \circ \beta^{s-n}_{n}.$$
Furthermore, in this expression the $r-m$ iterations of $\alpha_m$ and the $s-n$ iterations of $\beta_n$ may be shuffled amongst themselves in this expression (any order of these row- or column-doublings achieves the same effect).
\end{lem}
\begin{proof}
Each pre-composition with an $\alpha_m$ or a $\beta_n$ doubles the $n$th row or $m$th column of values of $f$. Since the $n$th row and $m$th column of $f$ are constant maps at $x_0$, this produces the same map as the trivial extension $\bar f$.
\end{proof}
 
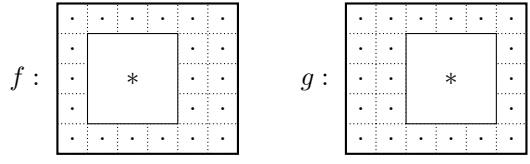
\begin{figure}
\[
f:
\vcenter{\hbox{
\begin{tikzpicture}[scale=.4]
\foreach \x in {0,...,5} {
 \foreach \y in {0,...,4} {
   \node at (\x+.5,\y+.5) {.};
  }
 }
\draw[thick] (0,0) rectangle (6,5);
\draw[densely dotted] (0,0) grid (6,5);
\draw[fill=white] (1,1) rectangle (4,4);
\node at (2.5,2.5) {$*$};
\end{tikzpicture}
}}
\qquad
g:
\vcenter{\hbox{
\begin{tikzpicture}[scale=.4]
\foreach \x in {0,...,5} {
 \foreach \y in {0,...,4} {
   \node at (\x+.5,\y+.5) {.};
  }
 }
\draw[thick] (0,0) rectangle (6,5);
\draw[densely dotted] (0,0) grid (6,5);
\draw[fill=white] (2,1) rectangle (5,4);
\node at (3.5,2.5) {$*$};
\end{tikzpicture}
}}
\]
\caption{Schematic of maps $f$ and $g$ from Corollary \ref{shiftcor}: A horizontal shift by one unit. \label{shiftfig}}
\end{figure}

\begin{cor}\label{shiftcor}
Let $f, g \colon (I_{m,n}, \partial I_{m,n}) \to (X,x_0)$ be continuous maps with $f(m-1,b)=x_0$ and $g(1,b)=x_0$ for each $j \in I_n$, and $g(a,b) = f(a-1, b)$ for each $a \in\{ 2, \ldots, m-1\}$ and each $j \in I_n$. That is, 
the maps $f$ and $g$ are \emph{horizontal shifts} of one another, as illustrated in Figure \ref{shiftfig}. Then we have $f\simeq g$.
\end{cor}
\begin{proof}
The assumption on $f$ means that $f$ is a trivial extension of some continuous map $h:I_{m-1,n} \to X$. In fact we have $f = h\circ \alpha_{m-1}$ and $g = h \circ \alpha_0$. By the proof of Theorem~\ref{columndoublehomotopy}, we have that $h\circ \alpha_{m-1} \simeq h\circ \alpha_0$, which implies that $f\simeq g$.
\end{proof}

\begin{rem}\label{shiftrem}
Repeated application of Corollary \ref{shiftcor} and Lemma~\ref{localized homotopy} allows for shifting horizontally up to homotopy, respectively extension homotopy, of a subrectangle of a function within a region of $I_{m,n}$ surrounded by basepoints.  A similar argument using row-doubling rather than column-doubling in Corollary~\ref{shiftcor} shows that the same is possible for vertical shifts. Combining these, we see that any ``translation'' of a subrectangle through a region of constant basepoint values will not change the homotopy class of a map.
 For example we may achieve up to homotopy any translation of the type appearing in Figure 4, where three blocks of values from $X$, each surrounded by basepoints, can be maneuvered into a different configuration within $I_{m,n}$.
\end{rem}

\begin{figure}
\[
\vcenter{\hbox{
\begin{tikzpicture}[scale=.5]
\draw[thick] (0,0) rectangle (11,9);
\draw[densely dotted] (0,0) grid (11,9);
\foreach \x in {.5,...,10.5} {
 \foreach \y in {.5,...,8.5} {
  \node[] at (\x,\y) {.};
  }
}
\draw[thick, fill=white] (1,1) rectangle (4,4) node[pos=.5] {$f_{R_1}$};
\draw[thick, fill=white] (3,5) rectangle (7,8) node[pos=.5] {$f_{R_2}$};
\draw[thick, fill=white] (7,1) rectangle (10,4) node[pos=.5] {$f_{R_3}$};
\end{tikzpicture}
}}
\simeq
\vcenter{\hbox{
\begin{tikzpicture}[scale=.5]
\draw[thick] (0,0) rectangle (11,9);
\draw[densely dotted] (0,0) grid (11,9);
\foreach \x in {.5,...,10.5} {
 \foreach \y in {.5,...,8.5} {
  \node[] at (\x,\y) {.};
  }
}
\draw[thick, fill=white] (1,5) rectangle (4,8) node[pos=.5] {$f_{R_1}$};
\draw[thick, fill=white] (5,5) rectangle (9,8) node[pos=.5] {$f_{R_2}$};
\draw[thick, fill=white] (1,1) rectangle (4,4) node[pos=.5] {$f_{R_3}$};
\end{tikzpicture}
}}
\]
\caption{Translations by homotopy of subrectangles surrounded by basepoint values illustrating Remark \ref{shiftrem}.\label{translationfig}}
\end{figure}
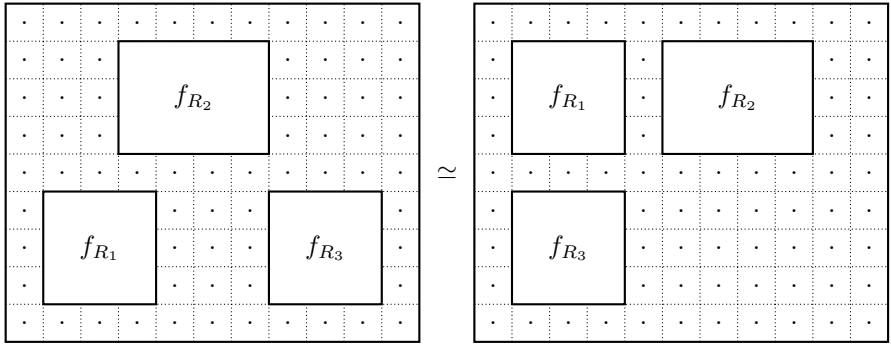

The following is a version of Lemma~\ref{localized homotopy} for ``local extension homotopy."

\begin{lem}\label{relativehomotopy}
Let  $R$ be a subrectangle of $I_{m,n}$. Let $f:(I_{m,n}, \partial I_{m,n}) \to (X, x_0)$ be a map for which $f_R$, the restriction of $f$ to the subrectangle $R$, is  a map of pairs $f_R: (R, \partial R) \to (X, x_0)$. Let $g:(R, \partial R) \to (X, x_0)$ be a map on the subrectangle such that $f_R\approx g$ by an extension homotopy in the sense of Definition~\ref{def: extn homotopy}.

Then the map $A: (I_{m,n}, \partial I_{m,n}) \to (X, x_0)$ defined by:
\[ A(x) = \begin{cases} f(x) &\text{ if } x\not \in R, \\
g(x) &\text{ if } x\in R \end{cases} \]
is continuous and we have $f \approx A$.
\end{lem}
\begin{proof}
Let $R = [r,s]_\Z \times [p,q]_\Z$. Since $f_R\approx g$, there is a larger rectangle $\bar R = [r,s+u]_\Z \times [p,q+v]_\Z$ with $u,v \ge 0$ and trivial extensions $\bar f_R, \bar g \colon \bar R \to X$ and a homotopy relative to the boundary $H_{\bar R}:\bar R\times I_k \to X$ from $\bar f_R$ to $\bar g$. By Lemma \ref{trivialdoubling} we have $\bar f_R = f_R \circ \alpha_s^u \circ \beta_q^v$ and $\bar g = g\circ \alpha_s^u \circ \beta_q^v$.

Now we define a homotopy $H:I_{m+u,n+v} \times I_k \to X$ to be a row- and column-doubled version of $f$ outside of $\bar R$, and equal to $H_{\bar R}$ inside of $R$:
\[ H(x,t) = \begin{cases} f\circ \alpha_s^u \circ \beta_q^v(x) &\text{ if } x \not \in \bar R, \\
H_{\bar R}(x,t) &\text{ if } x \in \bar R. \end{cases} \]

We will show that $H_R$ is a homotopy relative to the boundary from $f\circ \alpha_s^u \circ \beta_q^v$ to $A\circ \alpha_s^u \circ \beta_q^v$. Since $f \approx f\circ \alpha_s^u \circ \beta_q^v$ and $A \approx A\circ \alpha_s^u \circ \beta_q^v$ by Theorem \ref{columndoublehomotopy} and following, this will demonstrate that $f \approx A$ as desired.

We have: 
\[
H(x,0) = \begin{cases} f\circ \alpha_s^u \circ \beta_q^v(x) &\text{ if } x \not\in \bar R, \\
\bar f_R(x) &\text{ if } x\in \bar R \end{cases} = f\circ \alpha_s^u \circ \beta_q^v(x)
\]
and
\[
H(x,k) = \begin{cases} f\circ \alpha_s^u \circ \beta_q^v(x) &\text{ if } x \not\in \bar R, \\
\bar g(x) &\text{ if } x\in \bar R \end{cases} = A \circ \alpha_s^u \circ \beta_q^v(x)
\]
so $H$ begins at $f\circ \alpha_s^u \circ \beta_q^v$ and ends at $A \circ \alpha_s^u \circ \beta_q^v$.

Also it is easy to see that $H$ is a homotopy relative to the boundary: if $x\in \partial I_{m+u,n+v}$ then either $x\not\in \bar R$ or $x\in \partial \bar R$. In either case we have $H(x,t) = x_0$ for all $t$ because $f(\partial I_{m,n}) = x_0$ and $H_R$ is a homotopy relative to the boundary. Thus it remains only to show that $H$ is continuous.

Let $(x,t)\sim (x',t') \in I_{m+u,n+v} \times I_k$, and we must show that $H(x,t) \sim H(x',t')$. We prove this in simple cases according to whether or not the points $x,x'$ are in $\bar R$.

If $x\in \bar R$ and $x'\in \bar R$, then we have:
\[ H(x,t) = H_{\bar R}(x,t) \sim H_{\bar R}(x',t') = H(x',t') \]
where the middle step is because $H_{\bar R}$ is continuous. Thus $H(x,t) \sim H(x',t')$ as desired.

If $x\not \in \bar R$ and $x'\not\in \bar R$, then we have:
\[ H(x,t) = f\circ \alpha_s^u \circ \beta_q^v (x) \sim f\circ \alpha_s^u \circ \beta_q^v (x') = H_R(x',t') \]
where the middle step is because $f$, $\alpha_s$, and $\beta_q$ are continuous. Thus again $H(x,t) \sim H(x',t')$ as desired.

If $x\in \bar R$ and $x'\not \in \bar R$, this is only possible when $x \in \partial \bar R$. In that case we must have $f\circ \alpha_s^u \circ \beta_q^v(x) = x_0$ because $f$ maps $\partial R$ to $x_0$. Since $x \in \partial \bar R$ and $H_{\bar R}$ maps $\partial \bar R$ to $x_0$, we have:
\[ H(x,t) = H_{\bar R}(x,t) = x_0 = f\circ \alpha_s^u \circ \beta_q^v(x). \]
Since $x' \not \in \bar R$, we have $H(x',t') = f\circ \alpha_s^u \circ \beta_q^v(x')$. Since $f, \alpha_s,$ and $\beta_q$ are continuous and $x\sim x'$, this means
\[ H(x',t') = f\circ \alpha_s^u \circ \beta_q^v(x') \sim  f\circ \alpha_s^u \circ \beta_q^v(x) = H(x,t) \]
and so $H(x,t) \sim H(x',t')$,
which completes the proof that $H$ is continuous.
\end{proof}

%

The row- and column-doubling operations are closely related to \emph{subdivision} of a digital image, which was defined in \cite{EGS12} to define digitally continuous multivalued maps. The subdivision was used fundamentally in \cite{LOS19c} as the basis for the definition of a digital fundamental group.
In \cite{LOS19c} the authors use a general subdivision of a digital image which they denote $S(X,k)$ for some natural number $k$, which essentially replaces each point of $X$ by a $k\times k$ block of points. This subdivision comes with a natural projection map $\rho_k: S(X,k) \to X$ which collapses $k\times k$ blocks into single points. 

For our purposes we will only need to subdivide the rectangle $I_{m,n}$, in which case the $k$-fold subdivision is simply the rectangle $I_{km +k-1,kn+k-1}$. And then the projection map $\rho_k:I_{km+k-1,kn+k-1}\to I_{m,n}$ is obtained by iterated row and column omissions as follows:
\[ \rho_k = \alpha_0^{k-1} \circ \alpha_k^{k-1} \circ \dots \circ \alpha_{mk}^{k-1} \circ \beta_0^{k-1} \circ \dots \circ \beta_{kn}^{k-1}. \]


Applying Theorem \ref{columndoublehomotopy} repeatedly to the above gives:
\begin{thm}\label{subdivhomotopy}
Let $f:(I_{m,n}, \partial I_{m,n}) \to (X,x_0)$ be continuous and $k\ge 1$. Then $f \approx f \circ \rho_k$. \qed
\end{thm}

\section{Definition of $\pi_2(X,x_0)$}\label{sec: defn of pi_2}

We now give the definition of our second homotopy group and establish its basic general properties.

\begin{defn}
Given a based digital image $(X,x_0)$, the \emph{second homotopy group} of $(X,x_0)$, written $\pi_2(X,x_0)$, is the set of equivalence classes of maps $f:(I_{m,n},\partial I_{m,n}) \to (X,x_0)$, for all rectangles $I_{m,n}$, modulo the equivalence relation of extension homotopy.
\end{defn}

\begin{figure}
\[
\begin{tikzpicture}[scale=.4]
\foreach \x in {0,...,12} {
 \foreach \y in {0,...,10} {
   \node at (\x+.5,\y+.5) {.};
  }
 }
\draw[densely dotted] (0,0) grid (13,11);
\draw[thick] (0,0) rectangle (13,11);
\draw[thick] (0,6) -- (13,6);
\draw[thick] (6,0) -- (6,11);

\fill[white] (1,1) rectangle (5,5);
\node at (3,3) {$f$};

\fill[white] (7,7) rectangle (12,10);
\node at (9.5,8.5) {$g$};
\end{tikzpicture}
\]
\caption{Schematic of the product $f\cdot g$ of two maps $f$ and $g$. \label{cdotfig}}
\end{figure}
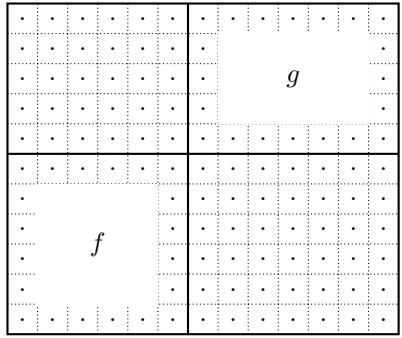

The group operation in $\pi_2(X,x_0)$ is induced by the following operation on maps. Let $f\colon (I_{m,n}, \partial I_{m,n})\to (X,x_0)$ and $g\colon (I_{r,s}, \partial I_{r,s})\to (X,x_0)$ be maps.  Define $f\cdot g\colon (I_{m+r+1, n+s+1}, \partial I_{m+r+1, n+s+1}) \to (X, x_0)$ by

$$
(f\cdot g)(a,b)=
    \begin{cases}
        f(a,b) & \text{if } (a,b)\in [0,m]_{\Z}\times [0,n]_{\Z}\\
        g(a-(m+1), b-(n+1)) & \text{if } (a,b)\in [m+1,m+r+1]_{\Z}\times [n+1,n+s+1]_{\Z}\\
        x_0 & \text{otherwise}.
    \end{cases}
$$
See Figure~\ref{cdotfig} for an illustration with $I_{m,n} = I_{5, 5}$ and  $I_{r,s} = I_{6, 4}$

\begin{prop}\label{prop: operation well-defined} Suppose we have maps 
$$f_1\colon (I_{m_1,n_1}, \partial I_{m_1,n_1})\to (X, x_0), \quad   f_2\colon (I_{m_2,n_2}, \partial I_{m_2,n_2})\to (X, x_0),$$
and
$$g_1\colon (I_{r_1,s_1}, \partial I_{r_1,s_1})\to (X, x_0), \quad g_2\colon (I_{r_2,s_2}, \partial I_{r_2,s_2})\to (X, x_0)$$
with $f_1\approx f_2$ and $g_1\approx g_2$. Then $f_1\cdot g_1\approx f_2 \cdot g_2$.
\end{prop}

\begin{proof} Let $f_1,f_2,g_1$, and $g_2$ satisfy the above. Then there are trivial extensions $\bar{f_1}, \bar{f_2}\colon I_{\bar{m},\bar{n}}\to X$ of $f_1$ and $f_2$, respectively, that are homotopic relative the boundary.  Similarly, there are trivial extensions $\bar{g_1}, \bar{g_2}\colon I_{\bar{r},\bar{s}}\to X$ of $g_1$ and $g_2$, respectively, that are homotopic relative the boundary. 

In the notation of Definition~\ref{def: alpha beta}, we may write $\bar{f_1}$ as $f_1 \circ \alpha_{m_1}^{\bar{m}-m_1} \circ \beta_{n_1}^{\bar{n}-n_1}$.  Since the first $m_1 +1$ columns and first $n_1+1$ rows (aside from extra $x_0$s to the right and above) of $f_1\cdot g_1$ are those of $f_1$, we may further write $\bar{f_1}\cdot g_1$ as $(f_1\cdot g_1) \circ \alpha_{m_1}^{\bar{m}-m_1} \circ \beta_{n_1}^{\bar{n}-n_1}$. Continuing to break down the trivial extensions into successive column- and row-doubling,  being careful with the indexing of rows and columns in the products, and using Theorem~\ref{columndoublehomotopy}, we have a sequence of extension homotopies as follows:
$$
\begin{aligned}f_1 \cdot g_1 &\approx (f_1\cdot g_1) \circ \alpha_{m_1}^{\bar{m}-m_1} \circ \beta_{n_1}^{\bar{n}-n_1} = \bar{f_1}\cdot g_1 \\
&\approx (\bar{f_1}\cdot g_1) \circ \alpha_{\bar{m}+r_1 +1}^{\bar{r}-r_1} \circ \beta_{\bar{n}+s_1+1}^{\bar{s}-s_1} = \bar{f_1}\cdot \bar{g_1}.
\end{aligned}
$$
Likewise, we may write $f_2 \cdot g_2 \approx \bar{f_2}\cdot \bar{g_2}$ by changing all subscripts from $1$ to $2$ in the above steps. 
Now the homotopies $\bar{f_1} \simeq \bar{f_2}$ and $\bar{g_1} \simeq \bar{g_2}$ extend to homotopies $\bar{f_1}\cdot \bar{g_1} \simeq \bar{f_2}\cdot \bar{g_1} \simeq \bar{f_2}\cdot \bar{g_2}$ 
by Lemma~\ref{localized homotopy}.  It follows that we have $f_1\cdot g_1 \approx f_2\cdot g_2$.
\end{proof}

Now let $[f],[g]\in \pi_2(X,x_0)$ and define $[f]\cdot[g]=[f\cdot g]$.  This operation is well-defined by Proposition \ref{prop: operation well-defined}.

\begin{thm}\label{thm: group} With the operation given above, the set of equivalence classes $\pi_2(X,x_0)$ is a group.
\end{thm}

\begin{proof}
Associativity follows immediately since  $(f\cdot g)\cdot h=f\cdot (g\cdot h)$ at the level of maps.

Next, let $c_{x_0}\colon I_{m,n}\to X$ be the constant map at $x_0 \in X$ from any rectangle.  Any such map may be viewed as a trivial extension of the constant map $c_{x_0}\colon I_{0,0}\to X$, where $I_{0,0} = \{ (0, 0) \}$.  We show that  $[c_{x_0}]$ acts as a two-sided identity, where---by the preceding remark---we may as well assume the representative $c_{x_0}$ has domain the single point $I_{(0, 0)}$.   Let $f\colon (I_{m,n}, \partial I_{m,n}) \to (X, x_0)$ be any map. On the right, we see that  
 $f\cdot c_{x_0}\colon I_{m+1, n+1} \to X$ is simply equal to the trivial extension $\bar f\colon I_{m+1, n+1} \to X$. Thus $[f]\cdot[c_{x_0}] = [f\cdot c_{x_0}] = [\bar f] = [f]$ and so $[c_{x_0}]$ acts on the right as an identity element.  On the left, we see that $c_{x_0}\cdot f\colon I_{m+1, n+1} \to X$ may be written as the result of doubling the first row and column of $f$, in the sense of Section~\ref{sec: subd and doubling}.  From 
 Theorem~\ref{columndoublehomotopy}, we may write
 $$c_{x_0}\cdot f = f\circ \alpha_0\circ\beta_0 \approx f\circ \alpha_0 \approx f.$$
 That is, we have $[c_{x_0}]\cdot [f] = [f]$ and $[c_{x_0}]$ acts as a left identity too.

Finally, we consider inverses.  For a map $f\colon I_{m,n}\to X$, define $f^{-1}\colon I_{m,n}\to X$ by $f^{-1}(a,b)=f(m-a,b)$.  As a pre-processing step,  we show that $f\cdot f^{-1}$ has the same extension homotopy type of the map that we denote by 
$(f\mid f^{-1})\colon (I_{2m+1,n}, \partial I_{2m+1,n}) \to (X, x_0)$ and define as 
$$
(f\mid f^{-1})(a, b)=
    \begin{cases}
        f(a,b) & \text{if } 0\leq a \leq m\\
        f^{-1}(a-(m+1),b) & \text{if } m+1 \leq a \leq 2m+1.\\
    \end{cases}
$$
To see this, note that on the sub-rectangle $R = [m+1,2m+1]_{\Z}\times [n+1,2n+1]_{\Z} \subseteq I_{2m+1, 2n+1}$, the map $f\cdot f^{-1}$ restricts to the map $f^{-1}\circ \beta_{0}^{n+1}$ in the notation of 
Theorem~\ref{columndoublehomotopy}.  By using translations of the type described in Remark \ref{shiftrem} and which flow from Corollary~\ref{shiftcor} and Lemma~\ref{localized homotopy}, $f^{-1}\circ \beta_{0}^{n+1}$ and $f^{-1}\circ \beta_{n+1}^{n+1}$ on the right-hand half of the rectangle are homotopic via a homotopy that extends to one of $f\cdot f^{-1}$ on the whole rectangle $I_{2m+1, 2n+1}$ and leaves the left-hand half fixed. This map may now be written as $(f\mid f^{-1})\circ \beta_{n+1}^{n+1}$ and repeated application of Theorem~\ref{columndoublehomotopy} yields an extension homotopy to $(f\mid f^{-1})$.  This pre-processing step may be summarized pictorially as a combination of translation and collapsing of repeated rows as follows:
\[
f\cdot g =
\vcenter{\hbox{\begin{tikzpicture}[scale=.4]
\draw[thick] (0,0) rectangle (2,3) node[pos=.5] {$f$};
\draw[thick] (2,3) rectangle (4,6) node[pos=.5] {$f^{-1}$};
\draw[thick] (0,3) rectangle (2,6) node[pos=.5] {$x_0$};
\draw[thick] (2,0) rectangle (4,3) node[pos=.5] {$x_0$};
\end{tikzpicture}
}}
\simeq
\vcenter{\hbox{\begin{tikzpicture}[scale=.4]
\draw[thick] (0,0) rectangle (2,3) node[pos=.5] {$f$};
\draw[thick] (2,3) rectangle (4,6) node[pos=.5] {$x_0$};
\draw[thick] (0,3) rectangle (2,6) node[pos=.5] {$x_0$};
\draw[thick] (2,0) rectangle (4,3) node[pos=.5] {$f^{-1}$};
\end{tikzpicture}
}}
\approx
\vcenter{\hbox{\begin{tikzpicture}[scale=.4]
\draw[thick] (0,0) rectangle (2,3) node[pos=.5] {$f$};
\draw[thick] (2,0) rectangle (4,3) node[pos=.5] {$f^{-1}$};
\end{tikzpicture}
}}
= (f\mid f^{-1})
\]

Now display the values of $(f\mid f^{-1})$ on the rectangle  $I_{2m+1, n}$ in column-wise form as $(f\mid f^{-1}) = [\mathbf{v}_0\mid \mathbf{v}_1\mid \cdots\mid \mathbf{v}_{m-1}\mid \mathbf{v}_m\mid \mathbf{v}_m\mid \mathbf{v}_{m-1}\mid \cdots\mid \mathbf{v}_0]$,
where each $\mathbf{v}_i$ is a column vector of entries from $X$ given by
$$\mathbf{v}_i = \left[ \begin{array}{c} f(a, n)\\ f(a, n-1)\\ \vdots \\ f(a, 1)\\f(a, 0) \end{array}\right].$$
Since the middle pair of columns repeat, we may write $(f\mid f^{-1}) = g_m\circ \alpha_m$, in the notation of Theorem~\ref{columndoublehomotopy}, where the values of $g_m$ on the rectangle  $I_{2m, n}$ in column-wise form are 
$$g_m = [\mathbf{v}_0\mid \mathbf{v}_1\mid \cdots\mid \mathbf{v}_{m-1}\mid \mathbf{v}_m\mid \mathbf{v}_{m-1}\mid \cdots\mid \mathbf{v}_0].$$
Namely, we have ``collapsed" the repeated $\mathbf{v}_m$ column into a single column and we have $(f\mid f^{-1}) \approx g_m$ by Theorem~\ref{columndoublehomotopy}.  Now define, for each $k = 0, \ldots, m$, a map
$g_k \colon (I_{2k, n}, \partial I_{2k, n}) \to (X, x_0)$ by
$$g_k(a, b) = \begin{cases} f(a, b) & a =0, \ldots, k\\ f(2k-a, b) & a= k+1, \ldots, 2k.\end{cases}$$
The map $g_m$ we arrived at above is the case in which $k = m$, and the general $g_k$ may be pictured column-wise as a reduced form of $g_m$, with 
$$g_k = [\mathbf{v}_0\mid \cdots\mid \mathbf{v}_{k-1}\mid \mathbf{v}_k\mid \mathbf{v}_{k-1}\mid \cdots\mid \mathbf{v}_0].$$
We note in this case that:
\[ g_{k-1}\circ \alpha_{k-1}^2 =  [\mathbf{v}_0\mid \cdots\mid \mathbf{v}_{k-1}\mid \mathbf{v}_{k-1}\mid \mathbf{v}_{k-1}\mid \cdots\mid \mathbf{v}_0]. \]

\emph{Claim}. For each $k \in \{m, m-1, \ldots, 1 \}$, we have $g_k \simeq g_{k-1}\circ \alpha_{k-1}^2 \approx g_{k-1}$, where the first homotopy is relative to the boundary.

\emph{Proof of Claim}. Repeated application of Theorem~\ref{columndoublehomotopy} gives an extension homotopy $g_{k-1}\circ \alpha_{k-1}^2 \approx g_{k-1}$, so we need only show that  $g_k \simeq g_{k-1}\circ \alpha_{k-1}^2$. In fact we will prove that $g_k \simeq g_{k-1}\circ \alpha_{k-1}^2$ by a one-step homotopy. By Lemma \ref{onesteplem}, take $(a,b) \sim (a',b')$, and we must show that $g_k(a,b) \sim g_{k-1}\circ \alpha_{k-1}^2(a',b')$.

Since $g_k$ and $g_{k-1}\circ \alpha_{k-1}^2$ differ only in column $k$, we need only consider the cases where $\{a,a'\} = \{k-1,k\}$. (The cases where $\{a,a'\} = \{k,k+1\}$ are similar.)

In the case where $a=k-1$ and $a'=k$, we have:
\[ g_{k-1}\circ \alpha^2_{k-1}(a',b') = g_{k-1}\circ \alpha^2_{k-1}(k,b') = g_k(k-1,b') \sim g_k(k-1,b) = g_k(a,b) \]
so $g_{k-1}\circ \alpha^2_{k-1}(a',b') \sim g_k(a,b)$ as desired.

In the case where $a=k$ and $a'=k-1$, we have:
\[ g_{k-1}\circ \alpha^2_{k-1}(a',b') = g_{k-1}\circ \alpha^2_{k-1}(k-1,b') = g_k(k,b') \sim g_k(k-1,b') = g_k(a',b')\]
so again $g_{k-1}\circ \alpha^2_{k-1}(a',b') \sim g_k(a,b)$ as desired. \emph{End of Proof of Claim}.  

The preceding arguments give a chain of extension homotopies as follows:
$$f\cdot f^{-1} \approx (f \mid f^{-1}) \approx g_m \approx g_{m-1} \approx \cdots \approx g_1 \approx g_0,$$
where the final map is  a constant map at $x_0$.  This shows that, for any map $f$, we have a right inverse $[f^{-1}]$ for $[f]$, which is sufficient for $[f^{-1}]$ to be a two-sided inverse for $[f]$. 
\end{proof}

\begin{rem}\label{turkishremark}
The earlier papers \cite{MVK14, VeKa15} define a higher digital homotopy group in a way that superficially uses the same ingredients that we do here. Some of the deductions in that work appear to be logically flawed, and a number of proofs are omitted. But the main difference between that work and ours stems from subtle but vital differences in the basic approach. In \cite{MVK14, VeKa15} a rectangle $I_{m,m}$ is assumed to have only $4$-adjacencies, which means that many more maps from a rectangle are admitted as continuous than are in our work.  Furthermore, the ``box" homotopy is used, which means that maps are more easily homotopic there than in our work here.  These differences result in totally different invariants.  For instance, the $2$-sphere that we use in this paper is easily shown to be contractible if the box homotopy is used in place of ours. All the higher homotopy groups of \cite{MVK14, VeKa15} would thus be trivial for our $2$-sphere and in fact no non-trivial example of a higher homotopy group is given in that work.
\end{rem}

The next result shows that the second homotopy group is independent of choice of basepoint.

\begin{prop}\label{prop: basepoint} If $x$ and $x'$ are in the same (path-connected) component of $X$, then $\pi_2(X,x) \cong \pi_2(X,x')$.
\end{prop}

\begin{proof}
Since $x$ and $x'$ are in the same component of $X$,  there is a sequence of adjacencies $x=x_0\sim x_1\sim \cdots \sim x_n=x'$.  By induction, it suffices to show that $\pi_2(X,x_0)$ is isomorphic to $\pi_2(X,x_1)$. For a map  $f\colon (I_{m,n}, \partial I_{m,n})\to (X,x_0)$, define a map $f_{x_1}\colon (I_{m+2,n+2}, \partial I_{m+2,n+2})\to (X,x_1)$ by
$$
f_{x_1}(a,b)=    \begin{cases}
        f(a-1,b-1) & \text{if } (a,b)\in [1,m+1]_{\Z}\times [1,n+1]_{\Z}\\
        x_1 & \text{otherwise }. \\
    \end{cases}
$$
If we picture maps from a rectangle as a labeling of points by their values in $X$, as we have before, this simply takes $f$ and surrounds it by a border of $x_1$. We claim that this induces a well-defined map $\Phi_{x_1}\colon \pi_2(X,x_0)\to \pi_2(X,x_1)$.    For suppose we have $[f] = [g] \in \pi_2(X,x_0)$, with $f\colon (I_{m,n}, \partial I_{m,n})\to (X,x_0)$ and $g\colon (I_{r,s}, \partial I_{r,s})\to (X,x_0)$.  Then there are trivial extensions $\bar{f}, \bar{g}\colon (I_{\bar{m},\bar{n}}, \partial I_{\bar{m},\bar{n}})\to (X,x_0)$ of each and a homotopy relative the boundary $H \colon I_{\bar{m},\bar{n}} \to X$ from $\bar{f}$ to $\bar{g}$.  By Lemma~\ref{localized homotopy}, this homotopy extends to a homotopy relative the boundary $H \colon I_{\bar{m}+2,\bar{n}+2} \to X$ from $(\bar{f})_{x_1}$ to $(\bar{g})_{x_1}$.  In the following sequence of extension homotopies, identifications and homotopies 
$$f_{x_1} \approx f_{x_1} \circ \alpha_{m+2}^{\bar{m}-m}\circ \beta_{n+2}^{\bar{n}-n} = (\bar{f})_{x_1} \simeq (\bar{g})_{x_1} = g_{x_1} \circ \alpha_{r+2}^{\bar{m}-r}\circ \beta_{s+2}^{\bar{n}-s} \approx g_{x_1},$$
the first and last extension homotopies follow from repeated application of Theorem~\ref{columndoublehomotopy}, the middle homotopy is the one we just observed, and the identifications follow from the definitions of the maps involved. Hence, we may define a map $\Phi_{x_1}\colon \pi_2(X,x_0)\to \pi_2(X,x_1)$ by setting $\Phi_{x_1}([f]) = [f_{x_1}]$. 
We show that $\Phi_{x_1}$ is an isomorphism.

To show that $\Phi_{x_1}$ is a homomorphism, we must show that  $[(f\cdot g)_{x_1}] = [f_{x_1}]\cdot [g_{x_1}]$ in $\pi_2(X,x_1)$, for $[f], [g]\in \pi_2(X,x_0)$.  Firstly, if $f\colon (I_{m,n}, \partial I_{m,n})\to (X,x_0)$ and $g\colon (I_{r,s}, \partial I_{r,s})\to (X,x_0)$, then we have an extension homotopy
$$(f\cdot g)_{x_1} \approx (f\cdot g)_{x_1}\circ \alpha_{m+1}^2 \circ \beta_{n+1}^2$$
from repeated application of Theorem~\ref{columndoublehomotopy}. For brevity, let $h = (f\cdot g)_{x_1}\circ \alpha_{m+1}^2 \circ \beta_{n+1}^2$. 

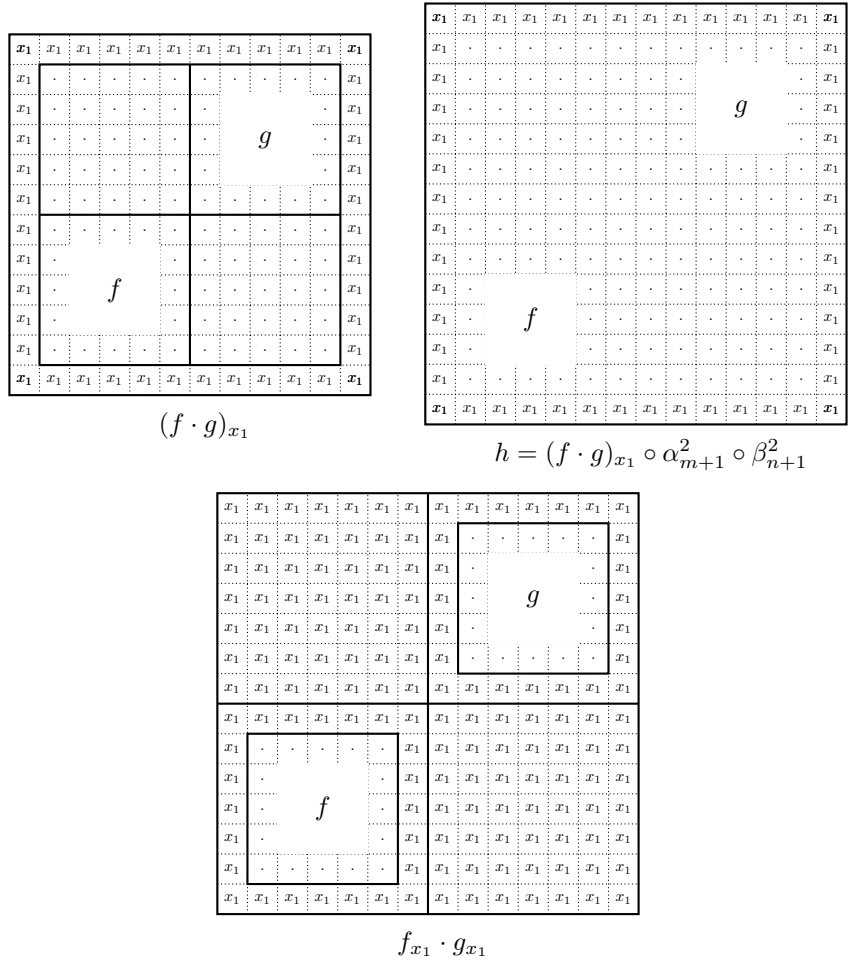
\begin{figure}
\newcommand{\nodescale}{.6}
\[
\vcenter{\hbox{\begin{tikzpicture}[scale=.4]
\foreach \x in {-1,...,10} {
 \foreach \y in {-1,10} {
  \node[scale=\nodescale] at (\x+.5,\y+.5) {$x_1$};
  }
  }
 \foreach \x in {-1,10} {
 \foreach \y in {-1,...,10} {
  \node[scale=\nodescale] at (\x+.5,\y+.5) {$x_1$};
  }}
\foreach \x in {0,...,9} {
 \foreach \y in {0,...,9} {
   \node[scale=\nodescale] at (\x+.5,\y+.5) {.};
  }
 }
\draw[densely dotted] (-1,-1) grid (11,11);
\draw[thick] (0,0) rectangle (10,10);
\draw[thick] (-1,-1) rectangle (11,11);
\draw[thick] (0,5) -- (10,5);
\draw[thick] (5,0) -- (5,10);
\fill[white] (1,1) rectangle node[black] {$f$} (4,4);
\fill[white] (6,6) rectangle node[black] {$g$} (9,9);
\node at (5.5,-2) {$(f\cdot g)_{x_1}$};
\end{tikzpicture}}}
\qquad
\vcenter{\hbox{\begin{tikzpicture}[scale=.4]
\foreach \x in {-1,...,12} {
 \foreach \y in {-1,12} {
  \node[scale=\nodescale] at (\x+.5,\y+.5) {$x_1$};
  }
  }
 \foreach \x in {-1,12} {
 \foreach \y in {-1,...,12} {
  \node[scale=\nodescale] at (\x+.5,\y+.5) {$x_1$};
  }}
\foreach \x in {0,...,11} {
 \foreach \y in {0,...,11} {
   \node[scale=\nodescale] at (\x+.5,\y+.5) {.};
  }
 }
\draw[densely dotted] (-1,-1) grid (13,13);
\draw[thick] (-1,-1) rectangle (13,13);
\fill[white] (1,1) rectangle node[black] {$f$} (4,4);
\fill[white] (8,8) rectangle node[black] {$g$} (11,11);
\node at (6.5,-2) {$h=(f\cdot g)_{x_1} \circ \alpha_{m+1}^2 \circ \beta_{n+1}^2$};
\end{tikzpicture}}}
\]
\[
\vcenter{\hbox{\begin{tikzpicture}[scale=.4]
\foreach \x in {-1,...,12} {
 \foreach \y in {-1,...,12} {
  \node[scale=\nodescale] at (\x+.5,\y+.5) {$x_1$};
  }
  }
\fill[white] (0,0) rectangle (5,5);
\foreach \x in {0,...,4} {
 \foreach \y in {0,...,4} {
   \node[scale=\nodescale] at (\x+.5,\y+.5) {.};
  }
 }
\fill[white] (7,7) rectangle (12,12);
\foreach \x in {7,...,11} {
 \foreach \y in {7,...,11} {
   \node[scale=\nodescale] at (\x+.5,\y+.5) {.};
  }
 }
\draw[densely dotted] (-1,-1) grid (13,13);
\draw[thick] (-1,-1) rectangle (13,13);
\draw[thick] (6,-1) -- (6,13);
\draw[thick] (-1,6) -- (13,6);
\draw[thick] (0,0) rectangle (5,5);
\draw[thick] (7,7) rectangle (12,12);
\fill[white] (1,1) rectangle node[black] {$f$} (4,4);
\fill[white] (8,8) rectangle node[black] {$g$} (11,11);
\node at (6.5,-2) {$f_{x_1} \cdot g_{x_1}$};
\end{tikzpicture}}}
\]
\caption{Various maps used in the proof of Proposition \ref{prop: basepoint}, which shows that $(f\cdot g)_{x_1} \approx (f\cdot g)_{x_1} \circ \alpha_{m+1}^2 \circ \beta_{n+1}^2 \simeq f_{x_1} \cdot g_{x_1}$. (Dots represent the base point $x_0$.)\label{basepointfig}}
\end{figure}

It suffices now to show that $h \simeq f_{x_1} \cdot g_{x_1}$. These maps are pictured in Figure \ref{basepointfig}. We see that $h$ and $f_{x_1} \cdot g_{x_1}$ differ only in certain points whose value under  $h$ is $x_0$, while the value under $f_{x_1} \cdot g_{x_1}$ is $x_1$. But all of these points have adjacent values only of $x_0$ or $x_1$. Thus we may perform repeated spider moves on $h$ which change all of these values of $x_0$ to $x_1$. In this way we obtain a homotopy $h \simeq f_{x_1} \cdot g_{x_1}$ as required.
\end{proof}

Next we show that given a continuous map of based digital images $\phi: (X,x_0) \to (Y,y_0)$, there is a natural  induced homomorphism on the fundamental group. This induced homomorphism is invariant under the following natural type of homotopy: We say that two based maps $\phi, \psi: (X,x_0) \to (Y,y_0)$ are \emph{homotopic relative to the basepoint} or \emph{based homotopic} when there is some continuous $H:X\times I_k \to Y$ such that $H$ is a homotopy from $\phi$ to $\psi$ and $H(x_0,t) = y_0$ for all $t\in I_k$.  

\begin{prop}\label{prop: functor}
If $\phi\colon (X,x_0)\to (Y,y_0)$ is a based digital map between based digital images,   there is an induced homomorphism $\phi_*\colon \pi_2(X,x_0)\to \pi_2(Y,y_0)$ given by $\phi_*([f])=[\phi\circ f]$.

This $*$ operator is functorial in the sense that $(\phi \circ \psi)_* = \phi_* \circ \psi_*$ for any based maps $\phi$ and $\psi$, and also $(\id_X)_* = \id_{\pi_2(X,x_0)}$, where $\id_X$ denotes the identity function of $X$.

Furthermore, if $\phi$ and $\psi$ are based homotopic, then $\phi_* = \psi_*$.
\end{prop}

\begin{proof} 
To show that $\phi_*$ is a group homomorphism, we observe:
$$\phi_*([f]\cdot[g])=\phi_*([f\cdot g])=[\phi \circ (f\cdot g)]=[(\phi\circ f)\cdot (\phi \circ g)]=[\phi\circ f]\cdot[\phi\circ g]=\phi_*([f])\cdot\phi_*([g]).$$

It is clear from our definitions that $(\psi\circ \phi)_*([f])=[\psi\circ \phi \circ f]=\psi_*[\phi\circ f]=\psi_*\circ \phi_*([f])$ and $(\id_X)_*([f])=[\id_X\circ f]=[f]=\id_{\pi_2(X,x_0)}[f]$, proving functoriality.

For the last statement, let $\phi, \psi \colon X\to Y$ be based homotopic, and we will show $\phi_*=\psi_*$.  Since $\phi$ and $\psi$ are homotopic, there exists a based homotopy $H\colon X\times I_k\to Y$,  Let $[f]\in \pi_2(X)$. To see that $[\phi\circ f]=[\psi\circ f]$, observe that $I_{m,n}\times I_k\to X\times I_k\to Y$ is a homotopy between $\phi\circ f$ and $\psi \circ g$. Because $H$ is a based homotopy, this homotopy of $\phi\circ f$ and $\psi \circ g$ is a homotopy relative to the boundary, and so $[\phi\circ f]=[\psi\circ f]$.
\end{proof}

Stated more abstractly, the above means that the second homotopy group $\pi_2$ is a functor from the category of based digital images and homotopy classes of digitally continuous maps to the category of abelian groups and group homomorphisms.

Two based digital images $(X,x_0)$ and $(Y,y_0)$ are \emph{based homotopy equivalent} when there are based maps $\phi:(X,x_0) \to (Y,y_0)$ and $\psi:(Y,y_0) \to (X,x_0)$ with $\phi \circ \psi$ and $\psi \circ \phi$ each based homotopic to identity maps on $(X,x_0)$ and $(Y,y_0)$ respectively.

\begin{thm}
Let $(X,x_0)$ and $(Y,y_0)$ be based homotopy equivalent. Then $\pi_2(X,x_0)$ and $\pi_2(Y,y_0)$ are isomorphic.
\end{thm}
\begin{proof}
Suppose we have based maps $\phi:(X,x_0) \to (Y,y_0)$ and $\psi:(Y,y_0) \to (X,x_0)$ with $\phi \circ \psi$ and $\psi \circ \phi$ each based homotopic to identity maps on $(X,x_0)$ and $(Y,y_0)$ respectively.  Then it follows from  Proposition~\ref{prop: functor} that we have  $\phi_* \circ \psi_* =  (\phi \circ \psi)_* = (\id_X)_* = \id_{\pi_2(X,x_0)}$ and likewise $\psi_* \circ \phi_* = \id_{\pi_2(Y,y_0)}$.  Hence, each of $\phi_*$ and $\psi_*$ must be an isomorphism.
\end{proof}

\begin{thm}\label{thm: abelian}
Given any pointed digital image $(X,x_0)$, the group $\pi_2(X,x_0)$ is abelian.
\end{thm}
\begin{proof}
The result follows using translations of the type described in Remark \ref{shiftrem} and which flow from Corollary~\ref{shiftcor} and Lemma~\ref{localized homotopy}. We  have $f\cdot g \simeq g\cdot f$ by homotopies indicated as follows: 
\[
\begin{split}
f\cdot g &=
\vcenter{\hbox{\begin{tikzpicture}[scale=.4]
\draw[thick] (0,0) rectangle (2,3) node[pos=.5] {$f$};
\draw[thick] (2,3) rectangle (6,5) node[pos=.5] {$g$};
\draw[thick] (0,3) rectangle (2,5) node[pos=.5] {$x_0$};
\draw[thick] (2,0) rectangle (6,3) node[pos=.5] {$x_0$};
\end{tikzpicture}
}}
\simeq
\vcenter{\hbox{\begin{tikzpicture}[scale=.4]
\draw[thick] (0,0) rectangle (2,3) node[pos=.5] {$f$};
\draw[thick] (0,3) rectangle (2,5) node[pos=.5] {$x_0$};
\draw[thick] (2,0) rectangle (6,2) node[pos=.5] {$g$};
\draw[thick] (2,2) rectangle (6,5) node[pos=.5] {$x_0$};
\end{tikzpicture}
}}
\simeq
\vcenter{\hbox{\begin{tikzpicture}[scale=.4]
\draw[thick] (0,0) rectangle (2,2) node[pos=.5] {$x_0$};
\draw[thick] (0,2) rectangle (2,5) node[pos=.5] {$f$};
\draw[thick] (2,0) rectangle (6,2) node[pos=.5] {$g$};
\draw[thick] (2,2) rectangle (6,5) node[pos=.5] {$x_0$};
\end{tikzpicture}
}}
\\
&\simeq
\vcenter{\hbox{\begin{tikzpicture}[scale=.4]
\draw[thick] (0,0) rectangle (4,2) node[pos=.5] {$g$};
\draw[thick] (0,2) rectangle (2,5) node[pos=.5] {$f$};
\draw[thick] (4,0) rectangle (6,2) node[pos=.5] {$x_0$};
\draw[thick] (2,2) rectangle (6,5) node[pos=.5] {$x_0$};
\end{tikzpicture}
}}
\simeq
\vcenter{\hbox{\begin{tikzpicture}[scale=.4]
\draw[thick] (0,0) rectangle (4,2) node[pos=.5] {$g$};
\draw[thick] (0,2) rectangle (4,5) node[pos=.5] {$x_0$};
\draw[thick] (4,0) rectangle (6,2) node[pos=.5] {$x_0$};
\draw[thick] (4,2) rectangle (6,5) node[pos=.5] {$f$};
\end{tikzpicture}
}}
= g\cdot f
\end{split}
\]
Recall that both $f$ and $g$ map their boundaries to the base point $x_0$. This means that, in the diagrams above, we may slide these blocks alongside each other without breaking continuity of the intermediate maps.
\end{proof}

The last general result that we give shows that our second homotopy group behaves with respect to products like the second homotopy group in the ordinary topological setting.  That is, our second homotopy group preserves products in the functorial sense.  Recall that, for based digital images $(X, x_0)$ and $(Y, y_0)$, their product $X \times Y$ is the categorical product and its basepoint is the point  $(x_0, y_0) \in X \times Y$.  The product of two (abelian) groups, denoted here by `$\times$,' means their direct product. This result uses the induced homomorphisms just discussed.

 \begin{thm}\label{thm: products}
Let $(X, x_0)$ and $(Y, y_0)$ be any based digital images.  Let $p_1\colon X \times Y \to X$ and $p_2\colon X \times Y \to Y$ denote the projections onto either factor.  Define a map
$$\Psi \colon  \pi_2\big(X \times Y; (x_0, y_0)\big) \to \pi_2(X; x_0)\times \pi_2(Y; y_0),$$
by setting $\Psi([\alpha]) := \big( (p_1)_*([\alpha]), (p_2)_*([\alpha])\big)$ for each $[\alpha] \in  \pi_2\big(X \times Y; (x_0, y_0)\big)$.
Then $\Psi$ is an isomorphism.
\end{thm}

\begin{proof}
Because $(p_1)_*$ and $(p_2)_*$ are both well-defined and homomorphisms, it follows that so too is $\Psi$  well-defined and a homomorphism. We show that $\Psi$ is both surjective and injective, and thus an isomorphism.

For surjectivity, suppose that we have $([\alpha], [\beta]) \in \pi_2(X; x_0)\times \pi_2(Y; y_0)$,  with  $\alpha\colon (I_{m,n}, \partial I_{m,n})\to (X,x_0)$ and $\beta\colon (I_{m',n'}, \partial I_{m',n'})\to (Y,y_0)$.  Then the maps $(\alpha, c_{y_0}) \colon I_{m, n} \to X \times Y$ and $(c_{x_0}, \beta) \colon I_{m', n'} \to X \times Y$ represent elements of $\pi_2\big(X \times Y; (x_0, y_0)\big)$.  We have
$$\Psi \big( [(\alpha, c_{y_0})] \cdot [(c_{x_0}, \beta)]\big) = \Psi \big( [(\alpha, c_{y_0})] \big) \Psi \big( [(c_{x_0}, \beta)]\big) = ([\alpha], [c_{y_0}]) ([c_{x_0}], [\beta]) =
([\alpha], [\beta]).$$
It follows that $\Psi$ is surjective.

For injectivity, suppose that we have $[\alpha] \in \pi_2\big(X \times Y; (x_0, y_0)\big)$ represented by a map $\alpha \colon I_{m, n} \to X \times Y$, such that $\alpha \in \ker \Psi$. That is, $\Psi([\alpha]) = ([c_{x_0}], [c_{y_0}]) \in \pi_2(X; x_0)\times \pi_2(Y; y_0)$.  Then $p_1\circ\alpha$ and $p_2\circ\alpha$ are extension-homotopic to constant maps.  
Suppose we have
$\overline{p_1\circ\alpha} \simeq c_{x_0}$ via a homotopy relative to the boundary  $H \colon I_{m', n'} \times I_T \to X$  and $\overline{p_2\circ\alpha} \simeq c_{y_0}$ via a homotopy relative to the boundary $G \colon I_{m'', n''} \times I_S \to Y$. The idea is simply to ``same size" the domains of these homotopies using a process akin to a ``$3$D trivial extension.".
Firstly, we  trivially extend $\overline{p_1\circ\alpha}$ and $\overline{p_2\circ\alpha}$ to a larger common domain  $I_{m''', n'''}$. We continue to write these new trivial extensions as  $\overline{p_1\circ\alpha}$ and $\overline{p_2\circ\alpha}$ . On any point $(i, j)$ of $I_{m''', n'''}$ not in $I_{m', n'}$, extend $H$ by setting $H(i,j, t) = x_0$ for all $t \in I_T$.  Likewise, extend $G$ to the stationary homotopy at $y_0$ on points of $I_{m''', n'''}$ not in $I_{m'', n''}$.  Since $H$ and $G$ were originally homotopies relative to the boundary (stationary at their respective basepoints), these extensions of the homotopies to the larger domains are evidently continuous.   Secondly, if $S \not=T$, then we may lengthen the shorter one by adding stationary steps at the basepoint.  Note that this depends on either homotopy ending at the constant map.  Thus, we have trivially extended our homotopies to  maps $H \colon I_{m''', n'''} \times I_R \to X$ and $G \colon I_{m''', n'''} \times I_R \to Y$, on some common domain.
Then we have a homotopy relative to the boundary
$$\big( H,  G \big) \colon  I_{m''', n'''} \times I_R \to X \times Y$$
from $\big( \overline{p_1\circ\alpha}, \overline{p_2\circ\alpha}\big)$ to $\big( c_{x_0}, c_{y_0}\big)$.  Now may write
$\alpha = ( p_1\circ\alpha, p_2\circ\alpha)$ and hence $\big( \overline{p_1\circ\alpha}, \overline{p_2\circ\alpha}\big)$ as a trivial extension $\overline{\alpha}$ of $\alpha$, and also
$c_{(x_0, y_0)} = \big( c_{x_0}, c_{y_0}\big)$.
it follows that we have $[\alpha] = [c_{(x_0, y_0)} ] \in \pi_2\big(X \times Y; (x_0, y_0)\big)$.  Hence $\Psi$ is injective and the result follows.
\end{proof}

%

\section{A triangle-counting function on $\pi_2(S^2)$}\label{sec: degree}

We now turn to the computation of  the second homotopy group of a digital two-sphere.
Various models have appeared in the literature of sphere-like digital images (see \cite{Evako2015} for instance). We will define the digital sphere $S^n$ as in classical geometry as the set of points in $(n+1)$-dimensional space at unit distance from the origin. In the digital setting, namely with points from $\Z^n$, this includes only points with a single nonzero coordinate of magnitude 1.

Specifically let $\be_i = (0,\dots, 0,1,0,\dots,0)$ be the $i$-th standard basis vector in $\Z^{n+1}$, and define \emph{the digital $n$-sphere}
\[ S^n := \{ \pm\be_1,\dots, \pm\be_{n+1} \}. \]
By $S^n$ we always mean this digital image of $2n$ points, rather than the classical manifold $S^n$. We always consider $S^n$ as a digital image with $c_n$ adjacency, that is, two points are adjacent when they differ by at most 1 in each of their coordinates.

From now on, we deal exclusively with $S^2$, the set of 6 points whose (non-self-) adjacencies form the octahedral graph:
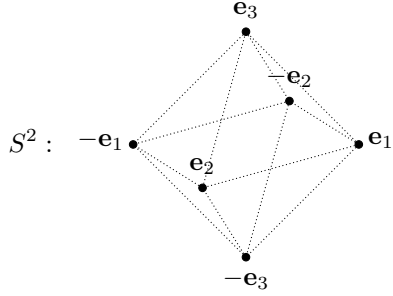
\tikzset{vertex/.style={circle,draw,fill,inner sep=0pt,minimum size=1mm}}
\[
S^2:
\vcenter{\hbox{
\begin{tikzpicture}[scale=1.5]
\node[vertex,label={[right]$\be_1$}] (p1) at (1,0,0) {};
\node[vertex,label={[left]$-\be_1$}] (m1) at (-1,0,0) {};
\node[vertex,label=$\be_3$] (p3) at (0,1,0) {};
\node[vertex,label={[below=.1cm]$-\be_3$}] (m3) at (0,-1,0) {};
\node[vertex,label=$\be_2$] (p2) at (0,0,1) {};
\node[vertex,label=$-\be_2$] (m2) at (0,0,-1) {};
\foreach \x in {p2,m2,p3,m3} {
 \draw[densely dotted] (p1) -- (\x) -- (m1);
}
\draw[densely dotted] (p2) -- (p3) -- (m2) -- (m3) -- (p2);
\end{tikzpicture}
}}
\]
For $a,b\in S^2$ with our chosen adjacency, we will have $a \sim b$ if and only if $a \neq -b$.  As discussed between Definitions \ref{def: continuous} and 
\ref{homotopydef},  we should really add loops at each vertex in the figure above to represent our $S^2$ strictly from the graph-theoretic point of view.  But we typically suppress these self-adjacencies from our figures.

The rest of the paper is concerned with showing that we have a group isomorphism $\pi_2(S^2, -\be_1) \cong \Z$.
The main tool we use for showing this is the following \emph{triangle-counting function} for maps from the rectangle into $S^2$.

We may view a map $f: (I_{m,n},\partial I_{m,n}) \to (S^2,-\mathbf{e}_1)$ as a labeling of the points of the rectangle $I_{m,n}$ with values from $S^2 = \{  \pm \mathbf{e}_1, \pm \mathbf{e}_2, \pm\mathbf{e}_3\}$.  Furthermore, we choose the triangulation of the rectangle $I_{m,n}$ that uses horizontal and vertical  edges (adjacencies)  between points, together with all positively sloped diagonal edges (adjacencies). (Here, we are setting aside all negatively sloped diagonal edges (adjacencies)  that also exist in the adjacency relation of $\Z^2$.) 

\begin{defn}[Triangle-counting function]
With the conventions above, define an integer $d(f)$ as the signed sum of the number of oriented triangles  labeled $\langle \mathbf{e}_1, \mathbf{e}_2, \mathbf{e}_3\rangle$ in this triangulation.  By oriented triangles, we mean that a triangle labeled $\langle \mathbf{e}_1, \mathbf{e}_2, \mathbf{e}_3\rangle$ in a counter-clockwise sense counts as $+1$ and a  triangle labeled $\langle \mathbf{e}_1, \mathbf{e}_2, \mathbf{e}_3\rangle$ in a clockwise sense counts as $-1$.
\end{defn}
\begin{figure}
\centering
$$
\begin{tikzpicture}[scale=1.3]

\node[inner sep=1.75pt, circle, fill=black] (v00) at (0,0) [draw] {};
\node[inner sep=1.75pt, circle, fill=black] (v01) at (0,1) [draw] {};
\node[inner sep=1.75pt, circle, fill=black] (v02) at (0,2) [draw] {};
\node[inner sep=1.75pt, circle, fill=black] (v03) at (0,3) [draw] {};
\node[inner sep=1.75pt, circle, fill=black] (v04) at (0,4) [draw] {};
\node[inner sep=1.75pt, circle, fill=black] (v10) at (1,0) [draw] {};
\node[inner sep=1.75pt, circle, fill=black] (v11) at (1,1) [draw] {};
\node[inner sep=1.75pt, circle, fill=black] (v12) at (1,2) [draw] {};
\node[inner sep=1.75pt, circle, fill=black] (v13) at (1,3) [draw] {};
\node[inner sep=1.75pt, circle, fill=black] (v14) at (1,4) [draw] {};
\node[inner sep=1.75pt, circle, fill=black] (v20) at (2,0) [draw] {};
\node[inner sep=1.75pt, circle, fill=black] (v21) at (2,1) [draw] {};
\node[inner sep=1.75pt, circle, fill=black] (v22) at (2,2) [draw] {};
\node[inner sep=1.75pt, circle, fill=black] (v23) at (2,3) [draw] {};
\node[inner sep=1.75pt, circle, fill=black] (v24) at (2,4) [draw] {};
\node[inner sep=1.75pt, circle, fill=black] (v30) at (3,0) [draw] {};
\node[inner sep=1.75pt, circle, fill=black] (v31) at (3,1) [draw] {};
\node[inner sep=1.75pt, circle, fill=black] (v32) at (3,2) [draw] {};
\node[inner sep=1.75pt, circle, fill=black] (v33) at (3,3) [draw] {};
\node[inner sep=1.75pt, circle, fill=black] (v34) at (3,4) [draw] {};
\node[inner sep=1.75pt, circle, fill=black] (v40) at (4,0) [draw] {};
\node[inner sep=1.75pt, circle, fill=black] (v41) at (4,1) [draw] {};
\node[inner sep=1.75pt, circle, fill=black] (v42) at (4,2) [draw] {};
\node[inner sep=1.75pt, circle, fill=black] (v43) at (4,3) [draw] {};
\node[inner sep=1.75pt, circle, fill=black] (v44) at (4,4) [draw] {};
\node[inner sep=1.75pt, circle, fill=black] (v50) at (5,0) [draw] {};
\node[inner sep=1.75pt, circle, fill=black] (v51) at (5,1) [draw] {};
\node[inner sep=1.75pt, circle, fill=black] (v52) at (5,2) [draw] {};
\node[inner sep=1.75pt, circle, fill=black] (v53) at (5,3) [draw] {};
\node[inner sep=1.75pt, circle, fill=black] (v54) at (5,4) [draw] {};

\draw (v00) -- (v50);
\draw (v01) -- (v51);
\draw (v02) -- (v52);
\draw (v03) -- (v53);
\draw (v04) -- (v54);
\draw (v00) -- (v04);
\draw (v10) -- (v14);
\draw (v20) -- (v24);
\draw (v30) -- (v34);
\draw (v40) -- (v44);
\draw (v50) -- (v54);
\draw (v03) -- (v14);
\draw (v02) -- (v24);
\draw (v01) -- (v34);
\draw (v00) -- (v44);
\draw (v10) -- (v54);
\draw (v20) -- (v53);
\draw (v30) -- (v52);
\draw (v40) -- (v51);

\foreach \i in {0, 1, 2, 3, 4, 5} {
\node[below]  at (\i,0) {$-\mathbf{e}_1$};
\node[above]  at (\i,4) {$-\mathbf{e}_1$};
}
\foreach \j in {1, 2, 3} {
\node[left]  at (0,\j) {$-\mathbf{e}_1$};
\node[right]  at (5,\j) {$-\mathbf{e}_1$};
}
\node[above left]  at (1,1) {$\mathbf{e}_3$};
\node[above left]  at (2,1) {$\mathbf{e}_3$};
\node[above left]  at (3,1) {$\mathbf{e}_3$};
\node[above left]  at (4,1) {$\mathbf{e}_3$};
\node[above left]  at (1,2) {$\mathbf{e}_2$};
\node[above left]  at (2,2) {$\mathbf{e}_1$};
\node[above left]  at (3,2) {$\mathbf{e}_1$};
\node[above left]  at (4,2) {$\mathbf{e}_2$};
\node[above left]  at (1,3) {$-\mathbf{e}_3$};
\node[above left]  at (2,3) {$\mathbf{e}_2$};
\node[above left]  at (3,3) {$\mathbf{e}_3$};
\node[above left]  at (4,3) {$\mathbf{e}_2$};

\node[scale=1.5]  at ((1.3, 1.7) {$+1$};
\node[scale=1.5]  at ((2.3, 2.7) {$-1$};
\node[scale=1.5]  at ((3.3, 2.7) {$+1$};
\node[scale=1.5]  at ((3.3, 1.7) {$-1$};

\end{tikzpicture}
$$
\caption{Triangulation of $I_{m,n}$ shown with $m=5$ and $n=4$.  Labelling with points from $S^2$ shows a map with $d(f) = 0$. }\label{fig: Example of map}
\end{figure}
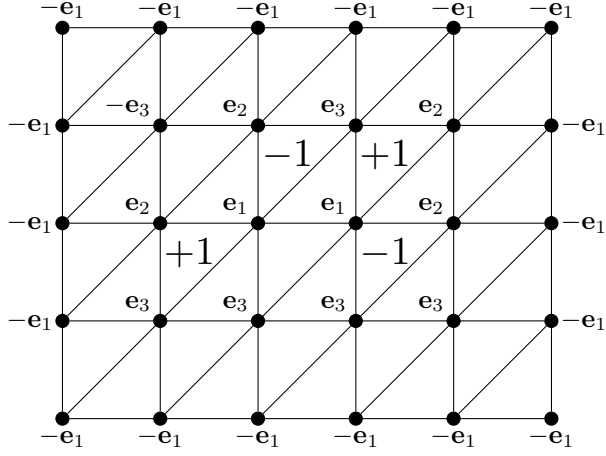
\begin{ex}
Take the map $f: (I_{5,4},\partial I_{5,4}) \to (S^2,-\mathbf{e}_1)$ whose values on the points of $I_{5,4}$ are as specified in Figure~\ref{fig: Example of map}.  We have triangulated $I_{5,4}$ in the way described above.  Note that, although we have not included diagonal edges of slope $-1$ in our triangulation, pairs of points that would be connected by such must be labeled with adjacent values from $S^2$ as well, to preserve continuity of the map $f$.  We find that there are two triangles labeled $\langle \mathbf{e}_1, \mathbf{e}_2, \mathbf{e}_3\rangle$ in a counter-clockwise sense and two in a clockwise sense, leading to a signed sum of $0$.  We will see that $d$ provides a function $d:\pi_2(S^2, -\mathbf{e}_1) \to \Z$, which we will eventually prove is a group isomorphism.
Thus, this map represents the trivial element  in $\pi_2(S^2, -\mathbf{e}_1)$, namely it must be extension-homotopic to the constant map $c_{-\mathbf{e}_1} \in S^2$.
\end{ex}

First we show that the triangle-counting function is preserved by extension-homotopy.

\begin{lem}
If $f\approx g$, then we have $d(f) = d(g)$. Thus, the integer $d(f)$ described above induces a well-defined \emph{induced triangle-counting function} $d: \pi_2(S^2, -\mathbf{e}_1) \to \Z$, given by setting $d([f]) = d(f)$.
\end{lem}

\begin{proof}
Let $f \colon (I_{m,n},\partial I_{m,n}) \to (S^2, -\mathbf{e}_1)$  and $g: (I_{m',n'},\partial I_{m',n'}) \to (S^2, -\mathbf{e}_1)$ be maps with $f\approx g$.  Recall that this means there are trivial extensions $\overline{f}$ of $f$ and $\overline{g}$ of $g$ such that $\overline{f}$ and $\overline{g}$ are defined on the same-sized rectangle as each other and are homotopic via a homotopy relative to the boundary.  Firstly, it is clear that we have $d(f) = d(\overline{f})$ and $d(g) = d(\overline{g})$, since a trivial extension preserves the labels of the original rectangle and simply labels additional points in the larger containing rectangle with $-\mathbf{e}_1$, thereby preserving all the original triangles labeled $\langle \mathbf{e}_1, \mathbf{e}_2, \mathbf{e}_3\rangle$ and not introducing any additional ones.  So, it is sufficient to show that we have $d(f)= d(g)$ when $f$ and $g$ are homotopic (and defined on the same-sized rectangle as each other).

By Theorem \ref{spidermovethm}, any homotopy can be effected by a sequence of spider moves which change only one point at a time. Thus, it suffices to show that $d(f) =d(g)$ when $f$ and $g$ are homotopic by a spider move that changes the values of $f$ at only one point.  So, we consider the situation in which $f, g \colon (I_{m,n},\partial I_{m,n}) \to (S^2, -\mathbf{e}_1)$ differ in value only at the point $(a,b) \in I_{m,n}$, for some $a$ and $b$ with $0<a<m$ and $0<b<n$ (recall that we  do not change the value of boundary points through our homotopy).

Since the  spider move only changes the map at the point $(a,b)$, with everything else remaining unchanged, the signed counts of triangles $d(f)$ and $d(g)$ may only differ according as the counts of triangles labeled $\langle \mathbf{e}_1, \mathbf{e}_2, \mathbf{e}_3\rangle$ differ in what we will call \emph{the hexagon} of labeled points in $I_{m, n}$, namely the points (labeled in their respective positions) illustrated in Figure~\ref{fig: hexagon}.
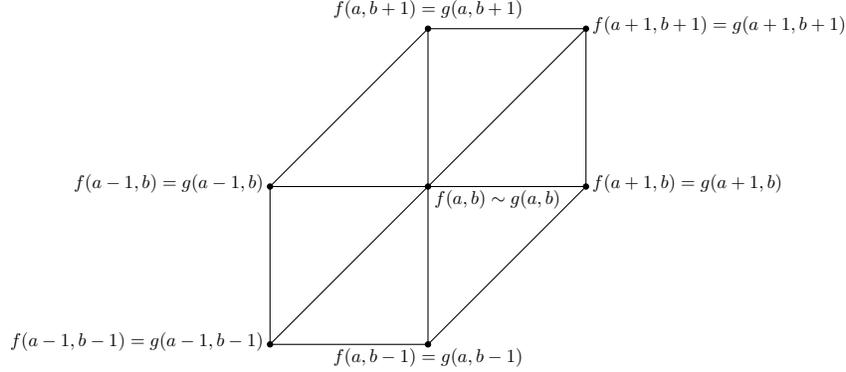
\begin{figure}
\[
\begin{tikzpicture}[scale=2.1,every node/.style={scale=.7}]
\node[vertex,label={[below right]$f(a,b)\sim g(a,b)$}] at (0,0) {};
\node[vertex,label={[right]$f(a+1,b+1)= g(a+1,b+1)$}] at (1,1) {};
\node[vertex,label={[above]$f(a,b+1)= g(a,b+1)$}] at (0,1) {};
\node[vertex,label={[left]$f(a-1,b)= g(a-1,b)$}] at (-1,0) {};
\node[vertex,label={[left]$f(a-1,b-1)= g(a-1,b-1)$}] at (-1,-1) {};
\node[vertex,label={[below]$f(a,b-1)= g(a,b-1)$}] at (0,-1) {};
\node[vertex,label={[right]$f(a+1,b)= g(a+1,b)$}] at (1,0) {};
\draw (1,0) -- (1,1) -- (0,1) -- (-1,0) -- (-1,-1) -- (0,-1) -- (1,0);
\draw (-1,0) -- (1,0);
\draw (0,-1) -- (0,1);
\draw (-1,-1) -- (1,1);
\end{tikzpicture}
\]
\caption{The hexagon of the $6$ triangles in $I_{m,n}$ surrounding $(a, b)$ labelled with points from $S^2$. }\label{fig: hexagon}
\end{figure}

If either $f(a,b)$ or  $g(a,b)$ has a value from $\{ -\mathbf{e}_1, -\mathbf{e}_2, -\mathbf{e}_3\}$ then the hexagon does not contribute any triangles labeled $\langle \mathbf{e}_1, \mathbf{e}_2, \mathbf{e}_3\rangle$ either before or after the spider move.  This follows because both $f(a,b)$ and  $g(a,b)$ must be adjacent to all values taken by points   of the hexagon---including both $f(a,b)$ and $g(a,b)$.  But if one of $f(a,b)$ or $g(a,b)$ takes a negative value $-\mathbf{e}_i$, then its positive counterpart $\mathbf{e}_i$ must be absent from the hexagon---again including both values $f(a,b)$ and $g(a,b)$.  However, we need all three values  $\{ \mathbf{e}_1, \mathbf{e}_2, \mathbf{e}_3\}$ to be taken on the hexagon by $f$ or $g$ in order to have any triangles to count at all.  Thus, in this case the
signed sums $d(f)$ and $d(g)$ are determined entirely by triangles in $I_{m,n}$ not involving the point $(a,b)$, and on which $f$ and $g$ agree.  That is, we have $d(f) = d(g)$.

It remains to consider the cases in which both of $f(a,b)$ and  $g(a,b)$ have (different) values from $\{ \mathbf{e}_1, \mathbf{e}_2, \mathbf{e}_3\}$.  Say $f(a,b) = \be_i$ and $g(a,b) = \be_j$. Let $k$ be the third value so that no two of $\{\be_i, \be_j, \be_k\}$ are equal. Without loss of generality, assume that $j = i+1 \mod 3$, which means that $k = i+2\mod 3$. Then the oriented count of $\langle \be_1,\be_2,\be_3\rangle$ triangles will equal the oriented count of $\langle \be_i,\be_j,\be_k\rangle$ triangles, so we can compute $d(f)$ and $d(g)$ by counting $\langle \be_i,\be_j,\be_k\rangle$ triangles.

Consider the values assigned to the $6$ points of the hexagon that surround $(a, b)$.  Since each of these points must be labelled with values in $S^2$ adjacent to both $\mathbf{e}_i$ and $\mathbf{e}_j$, all $6$ points must be labelled from amongst the points $\{ \mathbf{e}_i, \mathbf{e}_j, \pm \mathbf{e}_k \}\subseteq S^2$. If none of these $6$ points is labeled  $\mathbf{e}_k$, then the hexagon does not display any triangles labeled $\langle \mathbf{e}_i, \mathbf{e}_j, \mathbf{e}_k\rangle$  either before or after this spider move.  So, further assume that at least one of the $6$ points of the hexagon that surround $(a, b)$ is labeled  $\mathbf{e}_k$.  Starting at one such point,  travel counter-clockwise in a loop around the six vertices, listing the values with which they are labeled.  The result is a circuit $\gamma$ of length at most $6$
\begin{equation}\label{eq: length-6 circuit}
\gamma = (\mathbf{e}_k, v_1, v_2, v_3, v_4, v_5, \mathbf{e}_k)
\end{equation}
in the subgraph $\{ \mathbf{e}_i, \mathbf{e}_j, \pm \mathbf{e}_k\}$  of $S^2$, pictured with its adjacencies in Figure~\ref{fig: 6-point circuit}.
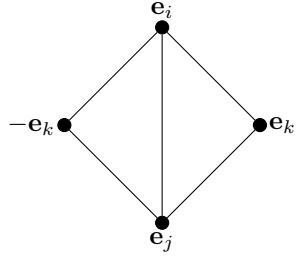
\begin{figure}
\centering

$$
\begin{tikzpicture}[scale=1.3]

\node[inner sep=1.75pt, circle, fill=black] (e3) at (0,0) [draw] {};
\node[inner sep=1.75pt, circle, fill=black] (e1) at (-1,1) [draw] {};
\node[inner sep=1.75pt, circle, fill=black] (f3) at (-2,0) [draw] {};
\node[inner sep=1.75pt, circle, fill=black] (e2) at (-1,-1) [draw] {};

\draw (e3) -- (e1);
\draw (e3) -- (e2);
\draw (e1) -- (e2);
\draw (e1) -- (f3);
\draw (e2) -- (f3);

\node[right]  at (e3) {{$\mathbf{e}_k$}};
\node[above]  at (e1) {{$\mathbf{e}_i$}};
\node[left]  at (f3) {{$-\mathbf{e}_k$}};
\node[below]  at (e2) {{$\mathbf{e}_j$}};

\end{tikzpicture}
$$
\caption{Subgraph of $S^2$ on which we have the circuit (\ref{eq: length-6 circuit}).  }\label{fig: 6-point circuit}
\end{figure}

Let $N(f)$ be the oriented count of $\langle \be_i, \be_j,\be_k\rangle$ triangles of $f$ occuring in the hexagon of Figure \ref{fig: hexagon}. Let $N(g)$ be the same count in the hexagon for $g$. Since $f$ and $g$ agree outside of the hexagon, we need only show that $N(f)=N(g)$.

Since $f$ maps the center point to $\be_i$, and $\gamma$ is oriented clockwise around the hexagon, the count $N(f)$ will equal the oriented count of the number of edges $(\be_j, \be_k)$ in $\gamma$. Similarly, the count $N(g)$ equals the oriented count of the number of edges $(\be_k,\be_i)$ in $\gamma$.

For some edge given as a vertex-pair $(v,v')$, let $K_\gamma (v,v')$ be the number of occurrences of the edge $(v,v')$ in $\gamma$. Then the above means that:
\begin{align*}
N(f) &= K_\gamma(\be_j,\be_k) - K_\gamma(\be_k,\be_j), \\
N(g) &= K_\gamma(\be_k,\be_i) - K_\gamma(\be_i,\be_k).
\end{align*}

Since $\gamma$ is a cycle based at $\be_k$, its in-degree at $\be_k$ must match its out-degree at $\be_k$. That is, $K_\gamma(\be_i,\be_k) + K_\gamma(\be_j,\be_k)  = K_\gamma(\be_k,\be_i) + K_\gamma(\be_k,\be_j)$. Combining this with the equations just displayed gives $N(f)-N(g) = 0$ and so $N(f)=N(g)$ as desired.
\end{proof}

\begin{prop}\label{prop: induced d a hom}
The induced triangle-counting function $d: \pi_2(S^2,-\be_1) \to \Z$ is a group homomorphism.  In particular, we have $d\left( [c_{-\be_1}]\right) = 0$ and $d\left( [f]^{-1}\right) =d\left( [f^{-1}]\right) = -d\left( [f]\right)$ for any $[f] \in \pi_2(S^2,-\be_1)$.
\end{prop}
\begin{proof}
Let $f,g:(I_{m,n},\partial I_{m,n}) \to (S^2, -\be_1)$ be two maps representing elements of $\pi_2(S^2,-\be_1)$. It is sufficient to show that $d(f\cdot g) = d(f) + d(g)$.

Recall that  $f\cdot g$ simply juxtaposes the grids defining $f$ and $g$ into a larger grid. Since $f$ and $g$ each individually map their boundary rectangles to the base point, there is no opportunity for the formation of new triangles involving $\be_1,\be_2,\be_3$ where the grids of $f$ and $g$ meet. Thus the total oriented count of $\langle \be_1,\be_2,\be_3\rangle$ triangles of $f\cdot g$ will equal the sum of the oriented counts for each of $f$ and $g$, which is to say $d(f\cdot g) = d(f)+d(g)$.

The last two assertions follow formally for any group homomorphism $d\colon G \to \Z$: with $G$ any group, we have $d(e) = 0$ and $d(g^{-1}) = - d(g)$ for any $g \in G$ and $e\in G$ the identity element.   The first of these is also easy to see directly: a constant map  $c_{-\be_1}$  contains no triangles labeled $\langle \mathbf{e}_1, \mathbf{e}_2, \mathbf{e}_3\rangle$.  The second is not so easy to see directly, as passing from $f$ to $f^{-1}$ involves re-labeling the points of a rectangle whilst preserving the triangulation---see the examples of $T, T^{-1} \colon (I_{4,4},\partial I_{4,4}) \to (X,-\be_1)$ given below. 
\end{proof}

We can now state the main result of this part of the paper: 

\begin{thm}\label{thm: d iso}
The induced triangle-counting function $d: \pi_2(S^2,-\be_1) \to \Z$ is a group isomorphism.
\end{thm}

The proof  of Theorem~\ref{thm: d iso} occupies the remainder of the paper. We will prove surjectivity immediately; the proof of injectivity requires some preparation.  

First, we introduce a specific map that will end up playing a prominent role in our calculation as a generator of  $\pi_2(S^2,-\be_1)$.

\begin{defn}\label{def: T}
Let $T: (I_{4,4}, \partial I_{4,4}) \to (S^2, -\be_1)$ be the map given by the labeling of points of $I_{4, 4}$ with values from $S^2$ as in Figure~\ref{fig: Defn ot T}.  In this diagram, we have used the style of earlier sections and indicated a label of the basepoint $-\be_1$ with a dot.  Furthermore, we have indicated labels of the three standard basis vectors with their subscript, and labels of $-\be_2$, respectively $-\be_3$, by $-2$, respectively $-3$.  We will adopt this style of diagram going forward.
\begin{figure}
\[
T: \drawisland{2}{3}{-2}{-3}
\]
\caption{The map $T\colon (I_{4,4},\partial I_{4,4}) \to (S^2, -\be_1)$ }\label{fig: Defn ot T}
\end{figure}
\end{defn}

\begin{lem}[\emph{Surjective part of Theorem~\ref{thm: d iso}}]\label{lem: d surjective}
The map $T: (I_{4,4}, \partial I_{4,4}) \to (S^2, -\be_1)$ given in Definition~\ref{def: T} satisfies $d(T)=1$.  Hence, the triangle-counting function $d: \pi_2(S^2,-\be_1) \to \Z$ is surjective.
\end{lem}
\begin{proof}
Figure~\ref{fig: Degree 1 map} shows the map $T$ with the triangulation we use to define the value of $d$.  In the figure, we see exactly one $\langle \be_1,\be_2,\be_3\rangle$ triangle, oriented in the positive direction. Thus we have $d([T])=1$. The consequence for surjectivity of $d$ follows immediately, since we already have shown that $d$ is a homomorphism.  Note that, \emph{per} Proposition~\ref{prop: induced d a hom}, we have 
$d([T^{-1}])=-1$ and $d\left( [c_{-\be_1}]\right) = 0$.  
\end{proof}

\begin{figure}
\centering
$$
\begin{tikzpicture}[scale=1.0,every node/.style={scale=.8}]

\foreach \i in {0, 1, 2, 3, 4}{
  \foreach \j in {0, 1, 2, 3, 4}{
    \node[inner sep=1.75pt, circle, fill=black] at (\i,\j) [draw] {}; } }

\foreach \i in {0, 1, 2, 3, 4}{
 \draw (\i, 0) -- (\i, 4);
 \draw (0, \i) -- (4, \i); }

\foreach \i in {0, 1, 2, 3}{
 \draw (0, \i) -- (4-\i, 4); }

\foreach \i in {1, 2, 3}{
 \draw (\i, 0) -- (4, 4-\i); }

\foreach \i in {0, 1, 2, 3, 4} {
\node[below]  at (\i,0) {$-\mathbf{e}_1$};
\node[above]  at (\i,4) {$-\mathbf{e}_1$};
}
\foreach \j in {1, 2, 3} {
\node[left]  at (0,\j) {$-\mathbf{e}_1$};
\node[right]  at (4,\j) {$-\mathbf{e}_1$};
}
\node[above left]  at (1,1) {$\be_3$};
\node[above left]  at (2,1) {$\be_3$};
\node[above left]  at (3,1) {$-\be_2$};
\node[above left]  at (1,2) {$\be_2$};
\node[above left]  at (2,2) {$\mathbf{e}_1$};
\node[above left]  at (3,2) {$-\be_2$};
\node[above left]  at (1,3) {$\be_2$};
\node[above left]  at (2,3) {$-\be_3$};
\node[above left]  at (3,3) {$-\mathbf{e}_3$};

\node[scale=1.5]  at ((1.3, 1.7) {$+1$};

\end{tikzpicture}
$$
\caption{The map $T\colon (I_{4,4},\partial I_{4,4}) \to (S^2, -\be_1)$ has $d(T) = 1$. }\label{fig: Degree 1 map}
\end{figure}

We now break off from the proof of Theorem~\ref{thm: d iso} to prepare for showing injectivity of the induced triangle-counting function. The proof
of Theorem~\ref{thm: d iso} is completed at the end of the next section.

\section{Islands, Flooding, and Injectivity of $d$}\label{sec: flooding}

We show several Lemmas that will be used in showing injectivity of $d$.  These Lemmas are summarized in Theorem~\ref{islandforms} below.
The first lemma is a simple application of spider moves.
\begin{lem}\label{2blocksisland}
Let $f:(I_{4,4},\partial I_{4,4}) \to (X,-\be_1)$ be continuous with $f(x)=\be_1$ for exactly one $x\in I_{4,4}$. Then $f$ is homotopic to a map of the form:
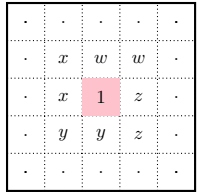
\[
\drawisland{x}{y}{z}{w}
\]
for $x,y,z,w \in \{\pm \be_2, \pm \be_3\}$.

Furthermore, if $\{x,y,z,w\} \neq \{ \pm \be_2, \pm \be_3\}$, then $f$ is homotopic to the constant map with constant value $-\be_{1}$.
\end{lem}
\begin{proof}
Because $f$ maps the boundary of $I_{4,4}$ to $-\be_1$, the only point which can map to $\be_1$ is the center point $(2,2)\in I_{4,4}$. Thus our map $f$ must take the following form:
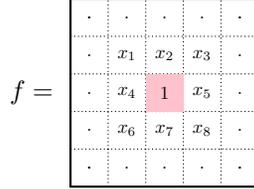
\[
f =
\drawislandeight{x_1}{x_4}{x_6}{x_7}{x_8}{x_5}{x_3}{x_2}
\]
for some $x_1,\dots,x_8 \in \{\pm \be_2,\pm \be_3\}$. We may do several spider move homotopies as in Lemma \ref{spidermove}. For example the value of $f(1,3) = x_1$ may be changed to $x_4$ because all neighbors of $(1,3)$ have labels which are already adjacent to $x_4$. Performing 4 similar spider moves results in the following map homotopic to $f$:
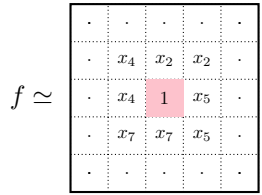
\[
f \simeq
\drawisland{x_4}{x_7}{x_5}{x_2}
\]
which has the desired format, proving the first statement of the Lemma.

For the second statement, assume that there is some $v \in \{\pm \be_2,\pm \be_3\}$ which is different from all values of $f$. Then all interior values of $f$ can be changed by spider moves to $-v$, and then to $-\be_1$:
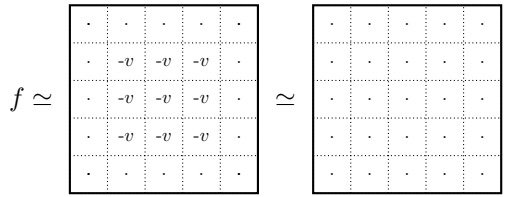
\[
f \simeq
\vcenter{\hbox{
\begin{tikzpicture}[scale=.5,every node/.style={scale=.7}]
\draw[thick, shift={(-.5,-.5)}] (0,0) rectangle (5,5);
\draw[densely dotted, shift={(-.5,-.5)}] (0,0) grid (5,5);
\foreach \x in {1,2,3} {
 \foreach \y in {1,2,3} {
  \node at (\x,\y) {-$v$};
 }
 }
\foreach \x in {0,...,4} {
 \node at (\x,0) {.};
 \node at (\x,4) {.};
 \node at (0,\x) {.};
 \node at (4,\x) {.};
}
\end{tikzpicture}
}}
\simeq
\vcenter{\hbox{
\begin{tikzpicture}[scale=.5,every node/.style={scale=.7}]
\draw[thick, shift={(-.5,-.5)}] (0,0) rectangle (5,5);
\draw[densely dotted, shift={(-.5,-.5)}] (0,0) grid (5,5);
\foreach \x in {0,...,4} {
 \foreach \y in {0,...,4} {
  \node at (\x,\y) {$.$};
 }
 }
\end{tikzpicture}
}}
\]
\end{proof}

Now the map $T$ introduced in Definition~\ref{def: T} has $T(x)=\be_1$ at exactly one point $x$ and satisfies $d(T) = 1$. The next lemma shows that these properties effectively characterize  $T$ up to extension homotopy.
\begin{lem}\label{islandTlem}
Let $f:(I_{4,4},\partial I_{4,4}) \to (X,-\be_1)$ be continuous with $f(x)=\be_1$ for exactly one $x\in I_{4,4}$ and $d(f)=1$. Then $f\approx T$.
\end{lem}
\begin{proof}
By Lemma \ref{2blocksisland} we may assume that $f$ takes the form:
\[
\drawisland{x}{y}{z}{w}
\]
where the values $x,y,z,w$ are taken from the set of possible values $\{\pm \be_2,\pm \be_3\}$. Since we assume that $d(f)=1$, this means that $f$ is not homotopic to a constant map, and so $\{x,y,z,w\} = \{\pm \be_2,\pm \be_3\}$ by the second part of Lemma \ref{2blocksisland}.

Since $(x,y,z,w)$ are all distinct, and we must have $x \sim y \sim z \sim w \sim x$, these values in order must be some cyclic permutation of $(\be_2,\be_3,-\be_2,-\be_3)$ or of $(\be_2, -\be_3,-\be_2,\be_3)$. In fact, only cyclic permutations of the first type will result in $d(f)=1$. The others will result in $d(f)=-1$, and so we do not consider them. Thus $f$ is homotopic to one of the following four maps:
\begin{equation}\label{fourmaps}
\begin{split}
\drawisland{2}{3}{-2}{-3}
\drawisland{-3}{2}{3}{-2}
\drawisland{-2}{-3}{2}{3}
\drawisland{3}{-2}{-3}{2}
\end{split}
\end{equation}
The first of these maps is $T$, so it will suffice to show that any of these maps is extension-homotopic to the others.

We do this by demonstrating a ``rotation of values'' by extension homotopies. Our demonstration is pictorial. In the sequence of diagrams below, each step is either a pair of consecutive spider moves or a row doubling:

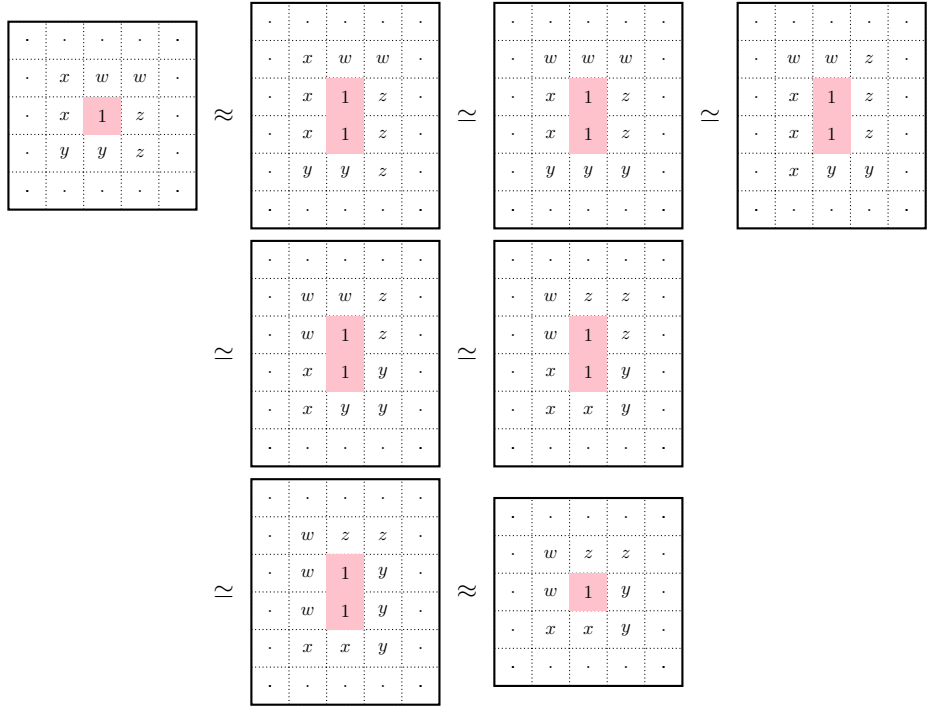
\begin{align*}
\drawisland{x}{y}{z}{w}
&\approx
\drawtallisland{x}{x}{x}{y}{y}{z}{z}{z}{w}{w}
\simeq
\drawtallisland{w}{x}{x}{y}{y}{y}{z}{z}{w}{w}
\simeq
\drawtallisland{w}{x}{x}{x}{y}{y}{z}{z}{z}{w} \\
&\simeq
\drawtallisland{w}{w}{x}{x}{y}{y}{y}{z}{z}{w}
\simeq
\drawtallisland{w}{w}{x}{x}{x}{y}{y}{z}{z}{z} \\
&\simeq
\drawtallisland{w}{w}{w}{x}{x}{y}{y}{y}{z}{z} 
\approx
\drawisland{w}{x}{y}{z}
\end{align*}

By repeatedly applying this rotation, we see that all four maps of \eqref{fourmaps} are extension-homotopic.
Since $f$ is homotopic to one of these four maps, and the first one is $T$, we have $f\approx T$.
\end{proof}

Let  $T^{-1} \colon (I_{4,4},\partial I_{4,4}) \to (X,-\be_1)$  be the inverse of $T$, as we defined inverses in the proof of Theorem~\ref{thm: group}. In diagrammatic terms, this is given as follows:
$$T^{-1} =
\drawislandeight{-3}{-2}{-2}{3}{3}{2}{2}{-3}$$
The following result corresponds to Lemma~\ref{islandTlem}  for maps with $d(f)=-1$ proved using the same arguments. We omit the details of its proof.
\begin{lem}
Let $f:(I_{4,4},\partial I_{4,4}) \to (X,-\be_1)$ be continuous with $f(x)=\be_1$ for exactly one $x\in I_{4,4}$ and $d(f)=-1$. Then $f\approx T^{-1}$. \qed
\end{lem}

We also consider such maps having $d(f)=0$.
\begin{lem}
Let $f:(I_{4,4},\partial I_{4,4}) \to (X,-\be_1)$ be continuous with $f(x)=\be_1$ for exactly one $x\in I_{4,4}$ and $d(f)=0$. Then $f$ is homotopic to a constant map.
\end{lem}
\begin{proof}
By Lemma \ref{2blocksisland} we may assume that $f$ takes the form:
\[
\drawisland{x}{y}{z}{w}
\]
where the values $x,y,z,w$ are taken from this set of 4 possible values: $\{\pm \be_2,\pm \be_3\}$. By the second part of Lemma \ref{2blocksisland}, if $\{x,y,z,w\} \neq \{\pm \be_2,\pm \be_3\}$ then $f$ is homotopic to a constant map and we are done. So it remains only to consider the case when $\{x,y,z,w\} = \{\pm \be_2,\pm \be_3\}$, and we will show that this case leads to a contradiction.

The proof of Lemma \ref{islandTlem} shows how the values $x,y,z,w$ can be ``rotated'' by extension homotopy. Thus we may assume without loss of generality that $x=\be_2$. Since $x$ appears adjacent to both $y$ and $w$, neither $y$ or $w$ can be $-\be_2$, and so we must have $z = -\be_2$. Thus either $y =\be_3$ and $w=-\be_3$, or $y=-\be_3$ and $w=\be_3$.

If $y = \be_3$ and $w = -\be_3$ , then $f = T$ which contradicts our assumption that $d(f)=0$, since $d(T)=1$. In the other case we have $w=-\be_3$ and $y=\be_3$, and so $f$ is a rotation of $T^{-1}$, which again contradicts our assumption that $d(f)=0$.
\end{proof}

We summarize the three lemmas above as follows:
\begin{thm}\label{islandforms}
Let $f:(I_{4,4},\partial I_{4,4}) \to (X,-\be_1)$ be continuous with $f(x)=\be_1$ for exactly one $x\in I_{4,4}$, and let:
\[ T =
\drawisland{2}{3}{-2}{-3},
\qquad
T^{-1} =
\drawislandeight{-3}{-2}{-2}{3}{3}{2}{2}{-3}
\]
If $d(f)=1$, then $f\approx T$. If $d(f)=-1$, then $f\approx T^{-1}$. If $d(f)=0$, then $f$ is homotopic to a constant map.
\end{thm}

Going forwards, we will refer to the $3 \times 3$ blocks of values with a value of $\be_1$ in the center as \emph{islands}.  We will eventually see that any map may be reduced, up to extension homotopy, to one whose values are represented by a number of these islands isolated from one another in a surrounding ``sea" of values of  $-\be_1$. 

Next, we introduce two more ingredients used to prove injectivity of $d$: a ``flooding" homotopy; and a pre-processing step that involves subdivision.

Given any map $f:(I_{m,n},\partial I_{m,n}) \to (S^2,-\be_1)$ and some particular value $b \in S^2$, we define a map $f_b:(I_{m,n},\partial I_{m,n}) \to (S^2,-\be_1)$ which we call \emph{the flood of $f$ with $b$} or \emph{the $b$-flood of $f$}, and which is a map that agrees in values with $f$ except that its values away from the boundary have been changed whenever possible to equal $b$. Specifically, we define:
\[ f_b(x) = \begin{cases}
b &\text{ if  $x\not\in \partial I_{m,n}$ and if $f(z) \neq -b$ for all $z\sim x$,} \\
f(x) &\text{ otherwise.}\end{cases} \]

Figure \ref{floodexamplefig} shows an example of a map $f$ with its flood by $-\be_{1}$.
\begin{figure}
\[
f: \vcenter{\hbox{
\begin{tikzpicture}[scale=.5,every node/.style={scale=.7}]
\foreach \x/\y in {2/2,6/2,6/3,6/4} {
\fill[shaded] (\x-.5,\y-.5) rectangle (\x+.5,\y+.5);
}
\foreach \x/\y/\l in {0/0/.,0/1/.,0/2/.,0/3/.,0/4/.,0/5/.,0/6/.,0/7/.,0/8/.,1/0/.,1/1/-2,1/2/-3,1/3/2,1/4/-3,1/5/-2,1/6/-2,1/7/-3,1/8/.,2/0/.,2/1/-2,2/2/1,2/3/2,2/4/.,2/5/-2,2/6/-2,2/7/-2,2/8/.,3/0/.,3/1/-2,3/2/3,3/3/3,3/4/3,3/5/3,3/6/3,3/7/.,3/8/.,4/0/.,4/1/-2,4/2/-2,4/3/3,4/4/.,4/5/3,4/6/2,4/7/2,4/8/.,5/0/.,5/1/-3,5/2/-2,5/3/3,5/4/2,5/5/2,5/6/2,5/7/-3,5/8/.,6/0/.,6/1/-3,6/2/1,6/3/1,6/4/1,6/5/-3,6/6/-3,6/7/.,6/8/.,7/0/.,7/1/-3,7/2/-3,7/3/-3,7/4/-3,7/5/-3,7/6/-3,7/7/-2,7/8/.,8/0/.,8/1/.,8/2/.,8/3/.,8/4/.,8/5/.,8/6/.,8/7/.,8/8/.} {
\node at (\x,\y) {\l};
}
\draw[densely dotted, shift={(-.5,-.5)}] (0,0) grid (9,9);
\draw[thick, shift={(-.5,-.5)}] (0,0) rectangle (9,9);
\end{tikzpicture}
}}
\qquad f_{-\be_{1}}: \vcenter{\hbox{
\begin{tikzpicture}[scale=.5,every node/.style={scale=.7}]
\foreach \x/\y in {2/2,6/2,6/3,6/4} {
\fill[shaded] (\x-.5,\y-.5) rectangle (\x+.5,\y+.5);
}
\foreach \x/\y/\l in {0/0/.,0/1/.,0/2/.,0/3/.,0/4/.,0/5/.,0/6/.,0/7/.,0/8/.,1/0/.,1/1/-2,1/2/-3,1/3/2,1/4/.,1/5/.,1/6/.,1/7/.,1/8/.,2/0/.,2/1/-2,2/2/1,2/3/2,2/4/.,2/5/.,2/6/.,2/7/.,2/8/.,3/0/.,3/1/-2,3/2/3,3/3/3,3/4/.,3/5/.,3/6/.,3/7/.,3/8/.,4/0/.,4/1/.,4/2/.,4/3/.,4/4/.,4/5/.,4/6/.,4/7/.,4/8/.,5/0/.,5/1/-3,5/2/-2,5/3/3,5/4/2,5/5/2,5/6/.,5/7/.,5/8/.,6/0/.,6/1/-3,6/2/1,6/3/1,6/4/1,6/5/-3,6/6/.,6/7/.,6/8/.,7/0/.,7/1/-3,7/2/-3,7/3/-3,7/4/-3,7/5/-3,7/6/.,7/7/.,7/8/.,8/0/.,8/1/.,8/2/.,8/3/.,8/4/.,8/5/.,8/6/.,8/7/.,8/8/.} {
\node at (\x,\y) {\l};
}
\draw[densely dotted, shift={(-.5,-.5)}] (0,0) grid (9,9);
\draw[thick, shift={(-.5,-.5)}] (0,0) rectangle (9,9);
\end{tikzpicture}
}}
\]
\caption{An example of a map $f:I_{8,8} \to S^2$ and the map $f_{-\be_{1}}$, the $-\be_{1}$-flood of $f$.\label{floodexamplefig}}
\end{figure}

\begin{lem}\label{floodhomotopic}
For $f:(I_{m,n},\partial I_{m,n}) \to (S^2,-\be_1)$ and $b\in S^2$, the flood of $f$ with $b$, $f_b:(I_{m,n},\partial I_{m,n}) \to (S^2,-\be_1)$ is continuous and homotopic to $f$ by a homotopy relative to the boundary.
\end{lem}
\begin{proof}
First we show that $f_b$ is continuous. Take $x\sim y\in I_{m,n}$, and we will show that $f_b(x)\sim f_b(y)$.
If $f(x)=-b$ or $f(y)=-b$, then by the definition of $f_b$ we will have $f_b(x)=f(x)$ and $f_b(y)=f(y)$ and so $f_b(x)\sim f_b(y)$ as desired. In the case where neither of $f(x)$ and $f(y)$ are $-b$, then $f_b(x) \in \{f(x),b\}$ and $f_b(y)\in \{f(y),b\}$. Since neither of $f(x)$ or $f(y)$ is $-b$, we have $f(x) \sim b$ and $f(y) \sim b$, and  $f_b(x)\sim f_b(y)$ as desired.

Now we show that $f_b$ is homotopic to $f$ in a single step. Take $x\sim y\in I_{m,n}$.  Since each of $f$ and $f_b$ are continuous, it is sufficient to show that $f_b(x)\sim f(y)$, by
Lemma~\ref{onesteplem}.
We investigate the various cases appearing in the definition of $f_b(x)$.

If $x \in \partial(I_{m,n})$ then $f_b(x) = -\be_1$, and also $f(x) = -\be_1$ and so $f_b(x) = f(x) \sim f(y)$ as desired.

If $f(z)\neq -b$ for all $z\sim x$, then in particular $f(y) \neq -b$, and $f_b(x) = b$. Since $f_b(x)=b$ and $f(y)\neq -b$, we have $f_b(x) \sim f(y)$ as desired.

The final case is when $x$ is outside the boundary and $f(z) = -b$ for some $z\sim x$. In this case $f_b(x) = f(x) \sim f(y)$ as in the first case.
\end{proof}

Next, we show how any $f\colon (I_{m,n},\partial I_{m,n}) \to (S^2,-\be_1)$ can be changed up to extension homotopy into a map in which the value $e_1$ occurs only at isolated points.

\begin{lem}\label{islandlem}
Let $f:(I_{m,n},\partial I_{m,n}) \to (S^2,-\be_1)$. Then there is a map $g \approx f$ having the property that no two adjacent points of the domain both map to $\be_1$. Furthermore, we may construct $g$ so that any two points $x,y$ in the domain of $g$ with $g(x)=g(y)=\be_1$ will be arbitrarily far apart.
\end{lem}
\begin{proof}
Choose some $k\ge 5$. We construct $g$ in three steps, beginning with $\hat f = f \circ \rho_k: S(I_{m,n},k) \to S^2$. This is the $k$-fold subdivision of $f$; we view it as a continuous function $\hat f\colon  (I_{km+k-1,kn+k-1},\partial I_{km+k-1,kn+k-1}) \to (S^2,-\be_1)$. By Theorem \ref{subdivhomotopy} we have $\hat f \approx f$.

For the second step we modify certain values of $\hat f$ in specific ways based on the specific locations of values with $\hat f(x)=\be_1$. Recalling that all values of $\hat f$ occur in constant $k\times k$ blocks, we make two adjustments described in Figures \ref{adjustment1} and \ref{adjustment2}. Making these two adjustments produces a new map $\hat f'$.

The first adjustment will change values of $\hat f$ in a small region surrounding two diagonally adjacent points $x \sim y$ with value $\hat f(x) = \hat f(y) = \be_1$, in which $x$ and $y$ are mutually adjacent to a pair of diagonal points with value different from $\be_1$. We adjust the values of $\hat f$ by making all points adjacent to either $x$ or $y$ have value $\be_1$, as in Figure \ref{adjustment1}.

\begin{figure}
\[ \vcenter{\hbox{\begin{tikzpicture}[scale=.3,every node/.style={scale=.7}]
\draw[fill= shaded] (0,5) rectangle (5,10);
\draw[fill= shaded] (5,0) rectangle (10,5);
\draw[densely dotted] (0,0) grid (10,10);
\draw[step=5] (0,0) grid (10,10);
\node[] at (4.5,5.5) {$x$};
\node[] at (5.5,4.5) {$y$};
\end{tikzpicture} }}
\to
\vcenter{\hbox{\begin{tikzpicture}[scale=.3,every node/.style={scale=.7}]
\draw[fill= shaded] (0,5) rectangle (5,10);
\draw[fill= shaded] (5,0) rectangle (10,5);
\fill[fill=shaded] (3,4) rectangle (7,6);
\fill[fill=shaded] (4,3) rectangle (6,7);
\draw[densely dotted] (0,0) grid (10,10);
\draw[step=5] (0,0) grid (10,10);
\node[] at (4.5,5.5) {$x$};
\node[] at (5.5,4.5) {$y$};
\end{tikzpicture} }}
\]
\caption{Adjustment \#1, in which all values adjacent to $x$ and $y$ become $\be_1$. In the picture, the shaded squares are pixels with value $\be_1$ under $\hat f$, and the white squares have value different from $\be_1$ (but not equal to $-\be_1$). In the picture we use $k=5$.\label{adjustment1}}
\end{figure}
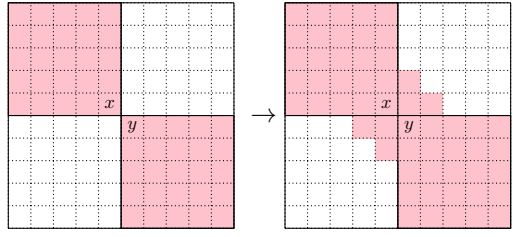

The second adjustment will change values of $\hat f$ in a small region near two orthogonally adjacent points $x\sim y$ with value $\hat f(x) = \hat f(y) = \be_1$, in which $x$ and $y$ arise from different pixels of $I_{m,n}$ before the subdivision, and each is adjacent to points with value different from $\be_1$. We adjust the values of $\hat f$ to make 2 other points adjacent to $x$ have value $\be_1$, as in Figure \ref{adjustment2}.

\begin{figure}
\[
\vcenter{\hbox{\begin{tikzpicture}[scale=.3,every node/.style={scale=.7}]
\draw[fill= shaded] (0,5) rectangle (10,10);
\draw[densely dotted] (0,0) grid (10,10);
\draw[step=5] (0,0) grid (10,10);
\node[] at (4.5,5.5) {$\vphantom{y}x$};
\node[] at (5.5,5.5) {$y$};
\end{tikzpicture} }}
\to
\vcenter{\hbox{\begin{tikzpicture}[scale=.3,every node/.style={scale=.7}]
\draw[fill= shaded] (0,5) rectangle (10,10);
\fill[fill=shaded] (3,4) rectangle (5,5);
\draw[densely dotted] (0,0) grid (10,10);
\draw[step=5] (0,0) grid (10,10);
\node[] at (4.5,5.5) {$\vphantom{y}x$};
\node[] at (5.5,5.5) {$y$};

\end{tikzpicture} }}
\]
\caption{Adjustment \#2, in which 2 values adjacent to $x$ become $\be_1$. In the picture, the shaded squares are pixels with value $\be_1$ under $\hat f$, and the white squares have value different from $\be_1$ (but not equal to $-\be_1$). In the picture we use $k=5$.\label{adjustment2}}
\end{figure}
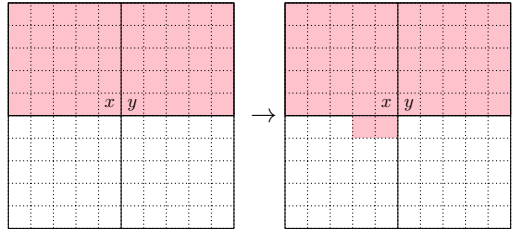

Applying these two adjustments in all applicable locations produces a map $\hat f' \colon  (I_{km+k-1,kn+k-1},\partial I_{km,kn}) \to (S^2,-\be_1)$.  Since each of these adjustments may evidently be effected by a (short) sequence of spider moves, we have $\hat f' \simeq \hat f$.

The third step of our construction is to consider $g = \hat f'_{\be_2,\be_3,-\be_1}$. This is the iterated flood of $\hat f'$ with $\be_2$ followed by $\be_3$ followed by $-\be_1$. By Lemma \ref{floodhomotopic} we will have $\hat f' \approx \hat f'_{\be_2,\be_3,-\be_1}$, and thus $g\approx f$ as desired. It remains to show that the values $x$ where $g(x)=\be_1$ are separated from one another.

The set $S$ of points $x$ with $\hat f'(x)=\be_1$ consists of a (possibly disconnected) union of $k\times k$ blocks of points, together with some extra points added by the two adjustments. This set looks something like Figure \ref{onesetfig}, in which $k=5$.

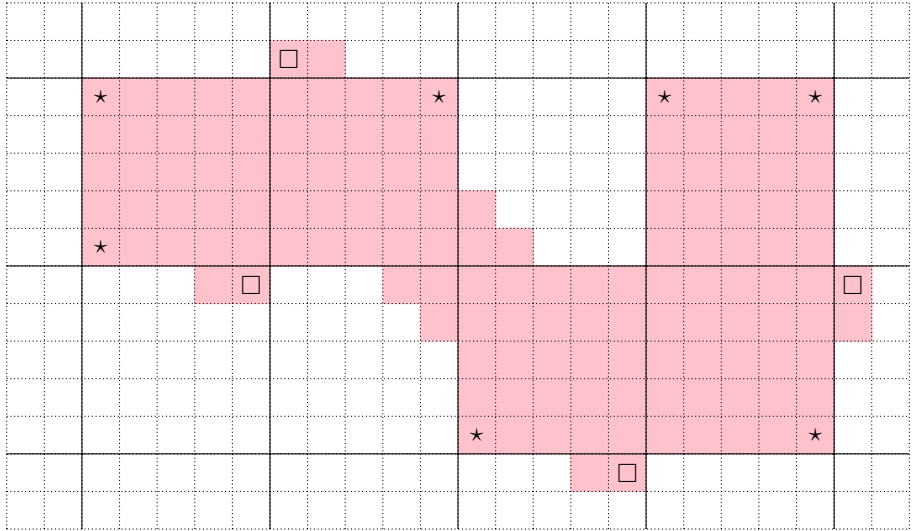
\begin{figure}
\[
\begin{tikzpicture}[scale=.5]
\filldraw[shaded] (0,5) rectangle (10,10);
\filldraw[shaded] (10,0) rectangle (20,5);
\filldraw[shaded] (15,5) rectangle (20,10);
\filldraw[shaded] (8,4) rectangle (12,6);
\filldraw[shaded] (9,3) rectangle (11,7);
\filldraw[shaded] (3,4) rectangle (5,5);
\filldraw[shaded] (5,10) rectangle (7,11);
\filldraw[shaded] (13,-1) rectangle (15,0);
\filldraw[shaded] (20,3) rectangle (21,5);
\draw[densely dotted] (-2,-2) grid (22,12);
\draw[step=5] (-2,-2) grid (22,12);
\foreach \p in {(0,5), (0,9),(9,9),(10,0),(15,9),(19,9),(19,0)} {
 \node[shift={(.25,.25)}] at \p {$\star$};
}
\foreach \p in {(4,4), (5,10), (14,-1), (20,4)} {
 \node[shift={(.25,.25)}] at \p {$\square$};
}
\end{tikzpicture}
\]
\caption{Possible arrangement of points with $\hat f'(x)=\be_1$. Points marked with $\square$ are ``adjusted corner'' points, which may be adjacent to other points with two different labels. Points marked with $\star$ are ``exterior corner'' points, which may be adjacent to other points with 3 different labels. All other points are adjacent to 0 or 1 labels other than $\be_1$. In the picture we use $k=5$.\label{onesetfig}}
\end{figure}

This set $S$ includes two types of special points: the ``exterior corner points'', lying on corners of the $k\times k$ blocks which are not adjacent to other $k\times k$ blocks in $S$, and the ``adjusted corner points", the points created by Adjustment \#2 which are adjacent to the neighboring $k\times k$ block. We label these points with $\star$ and $\square$ respectively in Figure \ref{onesetfig}. We will show that these exterior and adjusted corner points are the only points which can be labeled $\be_1$ by $g$. To do this, we show that all other points $x\in S$ have $g(x) \neq \be_1$. Since the exterior and adjusted corner points are never adjacent, and will be separated from each other by at least $k-2$ points in a horizontal or vertical direction, this will complete the proof.

Let $x\in S$ be some point which is not an exterior or adjusted corner of $S$. We consider cases according to the number of different values which can occur for points $\hat f'(y)$ with $y\sim x$. By our construction, since $x$ is not an exterior or adjusted corner, it will be adjacent to at most 1 other label different from $\be_1$.

If all points $y\sim x$ have the same value $\hat f'(y)=\be_1$, then $\hat f'_{\be_2}(x) = \be_2$ by the definition of the flood. Since no subsequent flood by $\be_3$ or $-\be_1$ can cause the label of $x$ to become $\be_1$, we see that $g(x) = \hat f'_{\be_2,\be_3,-\be_1}(x) \neq \be_1$ as desired.

Our second case is when $x$ is adjacent to only 1 label other than $\be_1$. That is, there is some $a\in S^2$ such that all points $y\sim x$ have either the value $\hat f'(y) = \be_1$ or $\hat f'(y) = a$. If $a\neq -\be_2$, then $\hat f'_{\be_2}(x) = \be_2$ and so as above we will have $g(x) \neq \be_1$ as desired. If $a = -\be_2$, then the flood by $\be_2$ will not change the label on $x$, so $\hat f'_{\be_2}(x) = \be_1$. But the subsequent flood by $\be_3$ will change the value to $\be_3$, giving $\hat f'_{\be_2,\be_3}(x) = \be_3$. Again as above this means that $g(x) \neq \be_1$ as desired.
\end{proof}

Finally, we are ready to complete the proof that $d: \pi_2(S^2,-\be_1) \to \Z$ is an isomorphism.

\begin{proof}[Proof of Theorem~\ref{thm: d iso}]
In Proposition~\ref{prop: induced d a hom} and Lemma~\ref{lem: d surjective} we have already shown that $d: \pi_2(S^2,-\be_1) \to \Z$ is a surjective homomorphism.  Here
we show that $\ker d$ is trivial. Let $f: (I_{m,n},\partial I_{m,n}) \to (S^2,-\be_1)$ with $d(f)=0$; we will show that $f \in \pi_2(S^2,-\be_1)$ is the trivial element.

By Lemma \ref{islandlem}, we may assume---up to extension homotopy---that $f$ has only isolated values of $\be_1$ occuring at the center of $3\times 3$ blocks of points, and outside of these $3\times 3$ ``islands'', $f$ is constant with value $-\be_1$. (The last flood by $-\be_1$ in the proof of Lemma~\ref{islandlem} will achieve the latter.) Furthermore we may assume that these islands are separated from each other by any distance we wish.

Using Lemma \ref{relativehomotopy} to apply Theorem \ref{islandforms} to each island locally, we may replace each island by the maps $T$, or $T^{-1}$ according to their triangle-count, or replace the island entirely by constant values $-\be_1$ if the triangle count for that island is zero.

Since these islands are separated by arbitrarily large regions of constant values of $-\be_1$, by Remark \ref{shiftrem} we may  translate the remaining islands into any configuration we wish. Since $d(f)=0$, the number of islands of type $T$ must equal the number of islands of type $T^{-1}$. Say that there are $k$ islands of each type. Then we may arrange them all to be stacked diagonally as in Figure \ref{cdotfig} (extending the domain if necessary), so that $f \approx k T \cdot k T^{-1}$, and this is the trivial element because $T$ and $T^{-1}$ are inverses.
\end{proof}

\section{Future work}\label{sec: future}

An obvious direction in which to continue is to define, for each $n$, a homotopy group $\pi_n(X, x_0)$ for $X$ a digital image.  We believe this should be a straightforward generalization of the approach taken here, with the group consisting of (suitably defined) extension-homotopy equivalence classes of maps $(I, \partial I) \to (X, x_0)$ with $I=I_{m_1} \times \cdots \times I_{m_n}$ an $n$-fold product of intervals.  We believe that most, if not all, of the results through Section~\ref{sec: defn of pi_2} should have direct generalizations.  One main issue in proceeding with this is simply an expositional one, dealing with the increasingly burdensome notational complexity.  It also seems reasonable to extend and generalize the development here to include  suitable relative homotopy groups and, if possible, develop the long exact sequence in homotopy groups in this digital context. 

In \cite{LuSc22} it was shown that the fundamental group (as defined there) of a 2D digital image is a free group.  A reasonable question to ask here is: must $\pi_2(X, x_0)$ for $X$ a 3D digital image be a free abelian group? 
On a related note, we may ask about torsion in $\pi_2(X, x_0)$.   For instance, what is an example of a digital image $X$ that has $\pi_2(X, x_0)$ with non-trivial torsion subgroup? Questions of this sort about $\pi_1(X, x_0)$ are effectively resolved in \cite{LuSc22} by the identification of the digital fundamental group with the ordinary fundamental group of the spatial realization of the clique complex of the digital image considered as a graph.  Here, then, we can ask whether $\pi_2(X, x_0)$ may be identified with the ordinary second homotopy group of the spatial realization of the clique complex of $X$? 

In a different direction, the calculation of  $\pi_2(S^2, -\be_1)$ given here  suggests many points of contact with classical topology.  Although our maps $I_{m,n} \to S^2$ are really just graph homomorphisms, the line of argument is strongly suggestive of topological ingredients such as the degree of a map, the classical homotopy group(s) of a topological space, triangulations, simplicial complexes, polyhedra and so-on.  Is it possible to somehow make these connections more precise?       


\providecommand{\bysame}{\leavevmode\hbox to3em{\hrulefill}\thinspace}
\providecommand{\MR}{\relax\ifhmode\unskip\space\fi MR }
\providecommand{\MRhref}[2]{%
  \href{http://www.ams.org/mathscinet-getitem?mr=#1}{#2}
}
\providecommand{\href}[2]{#2}

\end{document}